\documentclass[reqno, 12pt]{amsart} 

\usepackage[expansion=false]{microtype} 
\usepackage{amsfonts,amsthm,amsmath,amssymb,amscd,mathrsfs}
\usepackage{mathtools} 
\allowdisplaybreaks
\usepackage{latexsym} 
\usepackage[colorlinks=true,linkcolor=blue,citecolor=red,urlcolor=black]{hyperref} 
\usepackage{graphicx}
\usepackage{enumitem} 
\usepackage{indentfirst} 
\usepackage{cite} 
\usepackage{color} 
\usepackage{lmodern} 

\theoremstyle{plain} 
\newtheorem{Thm}{Theorem}[section] 
\newtheorem{Lem}[Thm]{Lemma}     
\newtheorem{Prop}[Thm]{Proposition}

\theoremstyle{definition}

\theoremstyle{remark}
\newtheorem{Rem}[Thm]{Remark}

\numberwithin{equation}{section} 

\newcommand{\beq}{\begin{equation}}
	\newcommand{\eeq}{\end{equation}}
\newcommand{\ben}{\begin{eqnarray}}
	\newcommand{\een}{\end{eqnarray}}
\newcommand{\beno}{\begin{eqnarray*}}
	\newcommand{\eeno}{\end{eqnarray*}}

\newcommand{\R}{\mathbb R}
\newcommand{\cD}{\mathcal D}
\newcommand{\cM}{\mathcal M}
\newcommand{\cF}{\mathcal F}
\newcommand{\supp}{\operatorname{supp}}
\newcommand{\diver}{\operatorname{div}}

\newcommand{\1}{\mathbf 1}

\newcommand{\eps}{\varepsilon}

\usepackage[
    letterpaper,                 
    textheight=8.35in,           
    headsep=20pt,                
    footskip=36pt,               
    marginparwidth=0pt,          
    marginparsep=0pt,            
    left=0.75in, 
    right=0.75in
]{geometry}

\usepackage[pagewise]{lineno}

\title[Critical Morrey Rigidity for Stationary Navier--Stokes Flows]{Critical Morrey Rigidity and Removable Singularities for Five-Dimensional Stationary Navier--Stokes Flows}

\author{
Yubo~Chen$^{1}$ \and
Wendong~Wang$^{1}$ \and
Xiao~Wang$^{1}$ \and
Guoxu~Yang$^{1}$ \and
Jianbo~Yu$^{1}$
}

\thanks{$^{1}$School of Mathematical Sciences, Dalian University of Technology, 
Dalian 116024, China.}

\thanks{E-mail addresses: 
1220823215@mail.dlut.edu.cn (Yubo~Chen),
wendong@dlut.edu.cn (Wendong~Wang),
xiaowang\_dlut@163.com (Xiao~Wang),
guoxu\_dlut@outlook.com (Guoxu~Yang),
yujb@mail.dlut.edu.cn (Jianbo~Yu).
}



\begin{document}
\raggedbottom

\begin{abstract}
We prove a critical Morrey rigidity theorem for the five-dimensional stationary Navier--Stokes equations. More precisely, every smooth solution on $\mathbb R^5\setminus\{0\}$ satisfying
\[
\sup_{R>0}R^{-2}\int_{B_R}|u|^3\,dx<\infty
\]
is identically zero, up to an additive constant in the pressure. This replaces the pointwise Type-I control in the known higher-dimensional rigidity theory by a velocity-only, scale-invariant averaged condition that allows spatial concentration. The proof develops a weak head-pressure mechanism that does not rely on pointwise pressure estimates or classical normal traces. We reconstruct a canonical pressure from the velocity, derive a renormalized inequality for the positive head pressure, and introduce two monotone radial fluxes. Annular energy estimates, suitable-weak compactness, and blow-up and blow-down limits are then used to identify the endpoint fluxes and force rigidity.

As an application, we obtain a removable-singularity criterion in dimension five: if a suitable weak solution is smooth away from one point and either its scale-invariant Dirichlet energy or its cubic velocity Morrey quantity remains bounded near that point, then the singularity is removable. Thus, within the isolated-singularity class, the smallness assumption in the classical stationary regularity criterion is replaced by boundedness. We also prove the corresponding velocity-only cubic Morrey rigidity theorem in dimension four by a different finite-energy argument.
\end{abstract}

\maketitle

\tableofcontents

\section{Introduction}

Consider the stationary incompressible Navier--Stokes equations
\begin{equation}
\left\{
\begin{aligned}
 -\Delta u+(u\cdot\nabla)u+\nabla p&=0,\\
 \diver u&=0
\end{aligned}
\right.
\qquad\text{in }G\subset\R^5.
\label{eq:NS}
\end{equation}
The five-dimensional stationary problem is the elliptic counterpart of the
three-dimensional nonstationary problem from the viewpoint of the energy scaling.  In
particular,
\[
 E_5(x_0,r):=r^{-1}\int_{B_r(x_0)}|\nabla u|^2\,dx
\]
is invariant under the natural transformation
\begin{equation}
 u^{(r)}(x)=r u(rx),
 \qquad p^{(r)}(x)=r^2p(rx).
 \label{eq:scaling}
\end{equation}

\subsection{Regularity criteria and bounded-scale removability}

The modern partial-regularity theory starts with Scheffer \cite{S1976} and Caffarelli--Kohn--Nirenberg
\cite{CKN1982}.  For suitable weak solutions of the three-dimensional evolutionary
system, they proved that the singular set has one-dimensional parabolic Hausdorff
measure zero in \cite{CKN1982}.  One standard form of their local
criterion says that there is a universal $\varepsilon_{\rm CKN}>0$ such that
\[
 \limsup_{r\downarrow0}r^{-1}
 \int_{Q_r(z_0)}|\nabla u|^2\,dx\,dt<\varepsilon_{\rm CKN}
\]
implies regularity at $z_0$.  New proofs, refinements, and related scale-invariant
criteria were subsequently developed by Lin \cite{Lin1998}, Ladyzhenskaya--Seregin
\cite{LadyzhenskayaSeregin1999}, Tian--Xin \cite{TianXin1999}, Seregin
\cite{Seregin2007}, Gustafson--Kang--Tsai \cite{GustafsonKangTsai2007}, Vasseur
\cite{Vasseur2007}, Kukavica \cite{Kukavica2008}, and Wang--Zhang
\cite{WangZhang2014}; see also the references therein.  This blow-up and local-energy
framework has guided the analysis of higher-dimensional stationary equations.

In dimension four, Gerhardt \cite{Gerhardt1979} obtained regularity of stationary weak
solutions.  Frehse and R\r{u}\v{z}i\v{c}ka developed existence and regularity theories
for stationary flows in dimensions five and six \cite{FrehseRuzicka1994}, while Li and
Yang later constructed regular stationary solutions on high-dimensional Euclidean
spaces \cite{LiYang2022}.

For the five-dimensional stationary problem, Struwe \cite{Struwe1988} proved partial
regularity and, in particular, regularity near $x_0$ under the smallness condition
\begin{equation} \label{eq:Struwe}
    \limsup_{r\downarrow0}E_5(x_0,r)<\varepsilon_0.
\end{equation}

Further regular-solution results on $\R^5$ were obtained in \cite{Struwe1995}, and the
boundary theory was developed by Kang \cite{Kang2004}.  Dimension six is the critical
endpoint for the standard energy method, since $H^1(\R^6)$ embeds into $L^3(\R^6)$
without compactness.  Dong--Strain \cite{DongStrain2012} established interior
partial regularity in this critical dimension, and Dong--Gu
\cite{DongGu2014} treated the boundary problem. The pressure is a principal difficulty in scale-invariant regularity criteria.  Liu--Wang \cite{LiuWang2018} proved six-dimensional boundary criteria in terms of the
velocity or Dirichlet energy and obtained a two-dimensional Hausdorff-measure estimate
for the boundary singular set.  Li--Wang \cite{LiWang2023} subsequently proved
interior and boundary one-scale criteria in dimension six; in particular, their
interior velocity criterion removes the pressure-smallness assumption by combining a
localized pressure decomposition with an inductive iteration.  In dimension five,
Cui \cite{C2023} obtained one-scale interior and boundary $\varepsilon$-regularity
criteria by a compactness/Campanato argument. These results show that pressure
smallness need not always be imposed as a separate hypothesis, but all of them remain
smallness criteria at the relevant scale.

The first aim of this paper is to remove this scale-invariant smallness requirement
within the isolated-singularity class.  Thus the solution is assumed to be smooth away
from one candidate point, but the critical quantities are only required to remain
bounded as the point is approached.  This is a genuine boundedness criterion, not an
$\varepsilon$-regularity statement; the punctured smoothness is the additional
structure that ultimately permits the use of a global tangent-flow rigidity theorem.

We use the following standard notion.  A pair $(u,p)$ is a suitable weak solution of
\eqref{eq:NS} in $G$ if
\[
 u\in H^1_{\rm loc}(G;\R^5),
 \qquad p\in L^{3/2}_{\rm loc}(G),
\]
the equations hold in $\cD'(G)$, and, for every nonnegative
$\phi\in C_c^\infty(G)$,
\begin{equation}
 \int_G|\nabla u|^2\phi\,dx
 \leq \int_G\frac{|u|^2}{2}\Delta\phi\,dx
 +\int_G\left(\frac{|u|^2}{2}+p\right)u\cdot\nabla\phi\,dx.
 \label{eq:LEI}
\end{equation}
For balls centered at the origin, set
\begin{equation}
 C_5(r):=r^{-2}\int_{B_r}|u|^3\,dx,
 \qquad
 D_5(r):=r^{-2}\int_{B_r}|p|^{3/2}\,dx,
 \qquad
 E_5(r):=r^{-1}\int_{B_r}|\nabla u|^2\,dx.
 \label{eq:CD}
\end{equation}
The pressure in a suitable pair is fixed once and for all.  Although $D_5$ changes
when a constant is added to $p$, this freedom causes no difficulty at small scales and
is explicitly recorded by one fixed-scale pressure norm in the pressure iteration.

Our first result is as follows.

\begin{Thm}
\label{thm:dirichlet-removability}
Let $(u,p)$ be a suitable weak solution of \eqref{eq:NS} in $B_1\subset\R^5$ and
assume that it is smooth in $B_1\setminus\{0\}$.  Suppose that, for some
$r_0\in(0,1)$, at least one of the following conditions holds:
\begin{align}
 &\sup_{0<r<r_0}E_5(r)<\infty,
 \label{eq:local-Dirichlet-bound}\\
 &\sup_{0<r<r_0}C_5(r)<\infty.
 \label{eq:local-velocity-Morrey-bound}
\end{align}
Then the origin is regular.  Consequently, $(u,p)$ is smooth in $B_1$.
\end{Thm}

\begin{Rem}
For the three-dimensional evolutionary problem, removing the smallness assumption from the general Caffarelli--Kohn--Nirenberg scale-invariant criterion remains open. In the axisymmetric setting, however, several critical pointwise and Type-I blow-up scenarios have been ruled out without a smallness assumption; see Koch--Nadirashvili--Seregin--Šverák\cite{KNSS2009} and Chen--Strain--Yau--Tsai\cite{CSYT2008}.
The removal of the smallness condition in \eqref{eq:Struwe} for the stationary five‑dimensional case, which is analogous to the non‑stationary three‑dimensional situation, is unknown.  In the isolated-singularity class, Theorem \ref{thm:dirichlet-removability} improves
the five-dimensional stationary criterion of Struwe \cite{Struwe1988} by replacing
smallness of the critical Dirichlet quantity with boundedness. It is reasonable to assume smoothness away from the origin, since the one‑dimensional Hausdorff measure of the singular set is zero.  The proof first controls
the scale-invariant mean $r|(u)_{B_r}|$, then obtains $C_5$ and contracts the harmonic
part of the pressure to obtain $D_5$.  Only after these two upgrades can one take a
nontrivial suitable-weak tangent and invoke global rigidity.
\end{Rem}

The proof of Theorem \ref{thm:dirichlet-removability} is in fact based on the following theorem.
For some fixed $r_0\in(0,1)$, write
\begin{equation}
 \sup_{0<r<r_0}C_5(r)\leq M,
 \qquad
 \sup_{0<r<r_0}D_5(r)\leq K.
 \label{eq:localMorrey}
\end{equation}

\begin{Thm}
\label{cor:first-singularity}
Let $(u,p)$ be a suitable weak solution in $B_1\subset\R^5$ which is smooth in
$B_1\setminus\{0\}$.  If \eqref{eq:localMorrey} holds, then the origin is regular and
$(u,p)$ is smooth in $B_1$.
\end{Thm}

Theorems \ref{thm:dirichlet-removability} and \ref{cor:first-singularity} are local
statements, but their decisive input is global.  If the origin were singular, a
scale-invariant blow-up would produce a nonzero global suitable-weak tangent with
bounded Morrey quantities.  The contradiction therefore reduces the removability
problem to a whole-space rigidity theorem, which leads to the second line of results.

\subsection{Rigidity of scale-invariant flows and global theorems}

A complementary development concerns rigidity of scale-invariant stationary flows.
In dimension three, the Landau solutions form an explicit nontrivial family of
$(-1)$-homogeneous solutions with one point singularity.  \v{S}ver\'ak
\cite{Sverak2011} proved that, up to the natural symmetries, they are the only smooth
$(-1)$-homogeneous solutions on $\R^3\setminus\{0\}$ and asked whether every smooth
three-dimensional solution satisfying $|U(x)|\le C|x|^{-1}$ must be a Landau solution.
Related descriptions of point singularities and far-field asymptotics can be found in
\cite{TianXin1998,MiuraTsai2012,KorolevSverak2011}.  The higher-dimensional analogue
has a different answer: Bang--Gui--Liu--Wang--Xie \cite{BGLWX2025} proved that every
smooth solution on $\R^n\setminus\{0\}$, $n\ge4$, satisfying the pointwise Type-I
estimate $|U(x)|\le C|x|^{-1}$ is trivial, without smallness of $C$ or a
self-similarity assumption.  They also applied this Liouville theorem to isolated
singularities and exterior-domain asymptotics.  Their proof exploits the total head
pressure and weighted energy multipliers.

The present paper connects this rigidity theory with the bounded-scale regularity
criteria above.  We replace pointwise Type-I control in dimension five by the integral
Morrey quantities $C_5,D_5,E_5$.  A centered cubic Morrey bound permits concentration
on thin annuli and does not imply $|U(x)|\le C|x|^{-1}$; consequently, pointwise
derivative estimates and $C^2_{\rm loc}$ compactness are no longer available.  The
proof must instead use suitable-weak compactness, canonical pressure reconstruction,
and monotone fluxes interpreted through annular averages.

The global result behind the regularity theorem requires only the velocity Morrey
bound.  For $n\geq4$ define
\[
 \cM_{3,n}(U):=\sup_{R>0}R^{3-n}\int_{B_R}|U|^3\,dx,
 \qquad \cM_3(U):=\cM_{3,5}(U),
\]
and, for a fixed pressure representative,
\begin{equation}
 \mathfrak M(U,P):=\cM_3(U)
 +\sup_{R>0}R^{-2}\int_{B_R}|P|^{3/2}\,dx.
 \label{eq:GM}
\end{equation}

Our third main result is as follows.

\begin{Thm}
\label{thm:smooth-punctured}
Let
\[
 (U,P_0)\in C^\infty(\R^5\setminus\{0\};\R^5\times\R)
\]
solve \eqref{eq:NS} in $\R^5\setminus\{0\}$ and assume only that
$\cM_3(U)<\infty$.  Then
\[
 U\equiv0,
 \qquad P_0\equiv c
\]
for some constant $c\in\R$.  More precisely, there is a unique constant $c$ such that
$P^\sharp:=P_0-c$ satisfies
\begin{equation}
 \sup_{R>0}R^{-2}\int_{B_R}|P^\sharp|^{3/2}\,dx
 \leq C\cM_3(U),
 \label{eq:intro-canonical-pressure}
\end{equation}
and for this canonical normalization $U=P^\sharp=0$.
\end{Thm}

\begin{Rem}
    The theorem of \cite{BGLWX2025} assumes the pointwise Type-I condition
$|U(x)|\leq C|x|^{-1}$ in every dimension $n\geq4$.  In dimension five, Theorem
\ref{thm:smooth-punctured} replaces this by the centered integral condition
$\cM_3(U)<\infty$.  The latter permits concentration on thin annuli and supplies neither pointwise derivative estimates nor uniform traces on every sphere. Their approach gives pointwise estimates of the head pressure $H:=\frac{1}{2} |u|^2 + p$: 
\begin{align*}
    |\nabla^k H(x)| \leq \frac{C_k}{|x|^{k+2}}, \,\, \text{in}\,\,\R^n\setminus\{0\},\,\, n\geq 4,\,\, k = 0,1,2\cdots
\end{align*}
while such bounds are difficult to obtain with  the Morrey estimates stated above.  This is
precisely where the pointwise argument ceases to apply: in dimension five its
positive-head-pressure flux contains
\[
 \int_{\partial B_R}\partial_rH_+\,H_+^{1/2}\,dS,
\]
whereas cubic Morrey control gives only bulk $L^{3/2}$ control of $H$ and no trace of
$\partial_rH$ controlled by the Morrey norm; after taking a suitable-weak limit,
there need not even be a classical normal trace.  One of our new observations is that the boundary estimate can be replaced by an annular Caccioppoli argument combined with a distributional flux. This approach further replaces pointwise endpoint estimates by blow‑up and blow‑down arguments. Together with the
canonical pressure reconstruction and suitable-weak compactness, this gives both a
weaker hypothesis in dimension five and a different endpoint mechanism (see Subsection 1.3 for more details).
\end{Rem}

\begin{Rem}
The normalization in Theorem~1.4 cannot be omitted. Indeed, \(U=0\) and \(P_0=c\neq0\) solve the equations and have zero velocity Morrey norm, whereas
\[
R^{-2}\int_{B_R}|c|^{3/2}\,dx\simeq |c|^{3/2}R^3\to\infty.
\]
Proposition~A.1 resolves this issue by reconstructing a canonical pressure \(P^\sharp\) directly from \(U\otimes U\). The local part is controlled by the \(L^{3/2}\)-boundedness of the Riesz transforms, while the far-field part is summable over dyadic annuli using the velocity Morrey bound. This yields
\[
\sup_{R>0}R^{-2}\int_{B_R}|P^\sharp|^{3/2}\,dx\le CM_3(U),
\]
and \(P_0-P^\sharp\) is necessarily constant; moreover, \(P^\sharp\) is the unique representative satisfying this Morrey bound. This differs from \cite{BGLWX2025}, where the pointwise Type-I assumption yields pointwise pressure estimates, in particular \(|P(x)|\lesssim |x|^{-2}\) after fixing the additive constant. Such a pointwise normalization is unavailable under the present cubic Morrey assumption, so the velocity-based pressure reconstruction is essential here.
\end{Rem}

The following one-scale criterion plays an important role in the proof, which is
\cite[Theorem~1.1]{C2023} with exponent $q=3$ and zero force; see also \cite[Theorem~1.2]{LiWang2023} for six dimension.

\begin{Prop}[One-scale $\eps$-regularity]
\label{thm:epsilon}
There exists a universal $\eps_*>0$ such that, if $(v,q)$ is a suitable weak solution
in $B_1\subset\R^5$ and
\[
 \int_{B_1}|v|^3\,dx<\eps_*,
\]
then $v$ is regular in $B_{1/2}$.
\end{Prop}


\begin{Rem}
For the three-dimensional nonstationary Navier--Stokes equations, Wolf \cite{W2015} used a local pressure decomposition to show that there exists an absolute constant \(\varepsilon_0>0\) such that
\[
r_0^{-2}\int_{Q_{r_0}}|u|^3\,dx\,dt\leq \varepsilon_0
\]
for some \(r_0>0\) implies regularity of \(u\) in \(Q_{r_0/2}\), without imposing a separate smallness condition on the pressure. See also Chae--Wolf \cite{CW2017} for related developments. Li--Wang \cite{LiWang2023} employed the pressure decomposition introduced by Wolf to establish regularity criteria for the six-dimensional stationary Navier--Stokes equations.
\end{Rem}

For the global suitable-weak argument, let
\[
 H:=\frac{|U|^2}{2}+P,
 \qquad W:=H_+^{3/4}.
\]
We say that $(U,P)$ has the \textbf{renormalized head-pressure property (RHP)} in
$G$ if $W\in H^1_{\rm loc}(G)$ and
\begin{equation}
 \diver\left(\frac43W\nabla W-\frac23UW^2\right)
 \geq\frac89|\nabla W|^2
 \quad\text{in }\cD'(G).
 \label{eq:RHP}
\end{equation}
The next propositions provide the key link between the RHP structure and the sign condition required for the rigidity argument.

\begin{Prop}
\label{Thm:RHP-sign}
Let $(U,P)$ be a global suitable weak solution satisfying
$\mathfrak M(U,P)<\infty$.  If $(U,P)$ has property $\mathrm{(RHP)}$ in
$\R^5\setminus\{0\}$, then $H\leq0$ almost everywhere.
\end{Prop}

\begin{Prop}
\label{thm:sign-rigidity}
Let $(U,P)$ be a global suitable weak solution in $\R^5$ satisfying
$\mathfrak M(U,P)<\infty$.  If
\begin{equation}
 H=\frac{|U|^2}{2}+P\leq0
 \quad\text{a.e. in }\R^5,
 \label{eq:sign-assumption}
\end{equation}
then $U=P=0$ in $\R^5$.
\end{Prop}

The last main theorem is a four-dimensional analogue proved by a different method.

\begin{Thm}
\label{thm:cubic-45}
Let
\[
 (U,P_0)\in C^\infty(\R^4\setminus\{0\};\R^4\times\R)
\]
solve the stationary Navier--Stokes equations in $\R^4\setminus\{0\}$.  Assume, for
the velocity alone, that
\begin{equation}
 \cM_{3,4}(U)=\sup_{R>0}R^{-1}\int_{B_R}|U|^3\,dx<\infty.
 \label{eq:Mncubic}
\end{equation}
Then $U=0$ and $P_0=c$ for a constant $c$.  The canonical normalization
$P^\sharp=P_0-c$ supplied by Proposition \ref{prop:auto-pressure} satisfies
\[
 \sup_{R>0}R^{-1}\int_{B_R}|P^\sharp|^{3/2}\,dx<\infty
\]
and is identically zero.
\end{Thm}

\begin{Rem}
\label{rem:cubic-high-dimensional-obstruction}
The velocity and pressure powers in the cubic formulation remain $3$ and $3/2$ in
every dimension, while Navier--Stokes scaling places the factor $R^{3-n}$ in front of
the integrals.  The present argument does not prove the corresponding assertion when
$n\geq6$.  Indeed, the positive-head-pressure part requires an $L^3$ estimate for
$W=H_+^{3/4}$.  The embedding
\[
 H^1(\R^n)\hookrightarrow L^{2n/(n-2)}(\R^n)
\]
is exactly critical for this purpose when $n=6$ and does not reach $L^3$ when $n>6$.
At the six-dimensional endpoint the resulting gradient term has no small factor and
cannot be absorbed under a mere boundedness assumption.  Thus $n\geq6$ remains open
for this cubic Morrey method; no failure of the rigidity statement itself is asserted.
\end{Rem}

\subsection{Main ideas and innovations}

The proof is organized around the following four observations.  The first two are the
new rigidity mechanisms; the last two connect them to the local regularity criteria
and make the main theorems velocity-only.

\medskip
\noindent\textbf{1. RHP, the normal-flux obstruction, and the monotone flux $\Xi_1$.}
For a smooth punctured solution the head-pressure equation gives RHP automatically (see Proposition \ref{prop:smooth-RHP}).
With $W=H_+^{3/4}$, define
\[
 \Xi_1(R)=\int_{\partial B_R}
 \left(\frac43W\nabla W-\frac23UW^2\right)\!\cdot e_r\,dS
\]
for almost every $R$.  Up to sign, this is the five-dimensional head-pressure flux
of \cite{BGLWX2025}.  Their proof controls the two classical boundary terms
\[
 \int_{\partial B_R}\partial_rH_+\,H_+^{1/2}\,dS,
 \qquad
 \int_{\partial B_R}U_rH_+^{3/2}\,dS
\]
by the scale-sharp bounds $|H(x)|\lesssim |x|^{-2}$,
$|\nabla H(x)|\lesssim |x|^{-3}$, and $|U(x)|\lesssim |x|^{-1}$.  A cubic Morrey bound yields only
$H\in L^{3/2}_{\rm loc}$ at the correct scale, so it supplies no corresponding
fixed-sphere trace estimate; after suitable-weak convergence the normal derivative
need not have a classical trace at all.  Directly repeating the pointwise proof is
also circular: flux control is used to prove $\nabla W\in L^2$, whereas the Sobolev
gain
$W\in L^{10/3}$ (equivalently $H_+\in L^{5/2}$) needed to make the endpoint flux
vanish is obtained only after that energy bound.

\textbf{One key observation:} by testing RHP on annuli, H\"older and the five-dimensional
annular Gagliardo--Nirenberg inequality absorb the critical term
$R^{-1}\int_{A_R}|U|W^2$ by a hole-filling iteration.  This gives a uniform annular
$L^2$ bound for $\nabla W$ and an averaged bound for $\Xi_1$ without a classical normal
trace.  Monotonicity then yields the global energy bound, Sobolev gives
$W\in L^{10/3}$, and a second averaged estimate makes both endpoint fluxes vanish.
Thus $\Xi_1\equiv0$ and $H\leq0$.  Strong annular pressure convergence also passes the
reduced RHP inequality to suitable-weak tangents.  This proves only the sufficient
implication $\mathrm{RHP}\Rightarrow H\leq0$, not that suitability alone implies RHP (see Subsection 2.4 for more details).

\medskip
\noindent\textbf{2. The monotone radial-defect flux $\Xi_2$ and its endpoint problem.}
Under $H\le0$ the exact identity
\[
 \Xi_2(R_1)-\Xi_2(R_2)
 =\int_{B_{R_2}\setminus B_{R_1}}
   \frac{3U_r^2-2H}{|x|^3}\,dx,
 \qquad
 \Xi_2(R)=R^{-2}\int_{\partial B_R}(U_r^2+P)\,dS,
\]
has a nonnegative defect.  Pointwise Type-I estimates give direct control of such a
flux on each sphere, but cubic Morrey bounds yield only the dyadic averaged estimate
\[
 \frac1R\int_R^{2R}|\Xi_2(s)|\,ds\leq C;
\]
this does not by itself imply that either endpoint value is zero.  The present proof uses monotonicity to obtain finite endpoint limits, which are then identified through blow-up at the origin and blow-down at infinity. By Proposition~\ref{prop:zero-defect}, each limiting zero-defect tangent is curl-free, divergence-free, and tangential to spheres, and is therefore trivial. Weak
pressure convergence then forces both endpoint fluxes to vanish, so $\Xi_2\equiv0$ and
$U=P=0$.

\medskip
\noindent\textbf{3. Local scale upgrading and nontrivial suitable-weak tangents.}
A bounded $E_5$ controls only the oscillation of $u$ at first.  Dyadic telescoping
controls $r|(u)_{B_r}|$ and hence $C_5$; a localized Calder\'on--Zygmund decomposition
and decay of the harmonic pressure then give
\[
 D_5(\theta r)\le C\theta^3D_5(r)+C\theta^{-2}C_5(2r).
\]
Thus either the $E_5$- or the $C_5$-hypothesis yields both compactness bounds.  If the
point were singular, the one-scale criterion keeps the suitable-weak global tangent
nonzero, while the first two observations force it to vanish.

\medskip
\noindent\textbf{4. Velocity-only pressure reconstruction and the dimension split.}
Appendix \ref{app:pressure} constructs a unique canonical $P^\sharp$ from $U\otimes U$
for $n=4,5$ and proves the correctly scaled $L^{3/2}$ Morrey bound.  A Bogovskii
cutoff removes the possible momentum defect at the puncture and shows
$P_0-P^\sharp$ is constant.  This reduction makes both main rigidity theorems
velocity-only.  The five-dimensional proof then uses RHP, $\Xi_1$, and $\Xi_2$, whereas
the four-dimensional proof closes through finite Dirichlet energy and
$\dot H^1(\R^4)\hookrightarrow L^4(\R^4)$.

The remainder of this paper is organized as follows. Section~2 develops the weak head-pressure framework. We establish scale-invariant compactness and pressure estimates, derive the renormalized head-pressure inequality for smooth solutions, and prove that the RHP structure forces the head pressure to be nonpositive. Section~3 proves the weak sign-rigidity theorem through the monotone radial-defect flux and the classification of zero-defect tangents. In Section~4, we combine these mechanisms with the canonical pressure reconstruction to prove the five-dimensional velocity-only critical Morrey rigidity theorem. Section~5 applies this rigidity result to isolated singularities and proves the boundedness regularity criteria. Section~6 treats the four-dimensional Morrey rigidity theorem by a separate finite-Dirichlet-energy argument. The appendix contains the velocity-based pressure reconstruction and the pressure-free annular estimates used throughout the paper.

\section{Preliminaries and the head-pressure sign mechanism}

This section develops the analytic framework used in both the rigidity and regularity
arguments.  We first obtain scale-invariant Caccioppoli bounds, suitable-weak
compactness, and strong pressure convergence on compact annuli.  We then derive RHP
for smooth solutions by a convex positive-part approximation, record the additional
renormalization needed in the weak class, and prove $H\leq0$ through annular estimates
for $W=H_+^{3/4}$ and the distributional monotone flux $\Xi_1$.  The essential new point
is that no classical normal trace on a prescribed sphere is used.

\subsection{Scale-invariant bounds and compactness}

\begin{Lem}
\label{lem:Cacc-u}
Let $(v,q)$ be a suitable weak solution to \eqref{eq:NS} in $B_{2R}\subset\R^5$ and suppose
\[
 R^{-2}\int_{B_{2R}}|v|^3\,dx\leq M_0,
 \qquad
 R^{-2}\int_{B_{2R}}|q|^{3/2}\,dx\leq K_0.
\]
Then
\begin{equation}
 R^{-1}\int_{B_R}|\nabla v|^2\,dx
 \leq C\left(M_0^{2/3}+M_0+M_0^{1/3}K_0^{2/3}\right).
 \label{eq:Cacc-u}
\end{equation}
\end{Lem}

\begin{proof}
Choose a cutoff $\eta$ supported in $B_{2R}$, equal to one in $B_R$, and satisfying
$|\nabla\eta|\leq C/R$ and $|\Delta\eta|\leq C/R^2$.  Insert $\eta$ in
\eqref{eq:LEI}.  The resulting terms are bounded by
\[
 CR^{-2}\int_{B_{2R}}|v|^2
 +CR^{-1}\int_{B_{2R}}|v|^3
 +CR^{-1}\int_{B_{2R}}|q||v|.
\]
H\"older's inequality gives
\[
 \int_{B_{2R}}|v|^2\leq CM_0^{2/3}R^3
\]
and
\[
 \int_{B_{2R}}|q||v|
 \leq
 \left(\int|q|^{3/2}\right)^{2/3}
 \left(\int|v|^3\right)^{1/3}
 \leq CK_0^{2/3}M_0^{1/3}R^2.
\]
Dividing the resulting estimate by $R$ proves \eqref{eq:Cacc-u}.
\end{proof}

\begin{Lem}
\label{lem:compactness}
Let $(v_k,q_k)$ be suitable weak solutions on an increasing sequence of balls exhausting
$\R^5$.  Assume that, for every $R>0$,
\[
 R^{-2}\int_{B_R}|v_k|^3\,dx\leq M,
 \qquad
 R^{-2}\int_{B_R}|q_k|^{3/2}\,dx\leq K
\]
for all sufficiently large $k$.  Then a subsequence converges to a global suitable weak
solution $(V,Q)$ such that, for every finite $R$,
\begin{align}
 v_k&\rightharpoonup V &&\text{in }H^1(B_R),\label{eq:weakH1}\\
 v_k&\to V &&\text{in }L^3(B_R),\label{eq:strongL3}\\
 q_k&\rightharpoonup Q &&\text{in }L^{3/2}(B_R).\label{eq:weakp}
\end{align}
The same Morrey bounds hold for $(V,Q)$.
\end{Lem}

\begin{proof}
Lemma \ref{lem:Cacc-u} gives a uniform local $H^1$ bound.  Since
$3<2^*=10/3$ in dimension five, Rellich's theorem yields the compact embedding
$H^1(B_R)\Subset L^3(B_R)$, proving \eqref{eq:strongL3} after diagonal extraction.
Weak compactness gives \eqref{eq:weakH1} and \eqref{eq:weakp}.  The nonlinear term
passes to the limit by strong $L^3$ convergence.  In the local energy inequality, the
gradient term is lower semicontinuous and
$q_kv_k\rightharpoonup QV$ in $L^1$, since $q_k$ is weak in $L^{3/2}$ and $v_k$ is
strong in $L^3$.  Thus the limit is suitable.  The Morrey bounds follow from strong
convergence for the velocity and weak lower semicontinuity for the pressure.
\end{proof}

\begin{Lem}
\label{lem:pressure-strong}
Let $G\subset\R^5$ be open.  Suppose $(v_k,q_k)$ solve the stationary
Navier--Stokes equations \eqref{eq:NS} in $G$ and
\[
 v_k\to V\quad\hbox{strongly in }L^3_{\rm loc}(G),
 \qquad
 q_k\rightharpoonup Q\quad\hbox{weakly in }L^{3/2}_{\rm loc}(G).
\]
Then, after passage to a subsequence,
\[
 q_k\to Q\quad\hbox{strongly in }L^{3/2}_{\rm loc}(G).
\]
\end{Lem}

\begin{proof}
Fix concentric balls
\[
 B\Subset B_0\Subset B_1\Subset G
\]
and choose $\chi\in C_c^\infty(B_1)$ with $\chi=1$ on a neighborhood of
$\overline{B_0}$.  Extend $\chi v_{k,i}v_{k,j}$ by zero and set on $\R^5$
\[
 q_k^{(1)}:=\mathcal R_i\mathcal R_j
 \bigl(\chi v_{k,i}v_{k,j}\bigr),
\]
where $\mathcal R_i$ are the Riesz transforms and repeated indices are summed.  Since
$v_k\to V$ strongly in $L^3(B_1)$,
\[
 \chi v_{k,i}v_{k,j}\longrightarrow \chi V_iV_j
 \quad\hbox{strongly in }L^{3/2}(\R^5).
\]
The Calder\'on--Zygmund theorem therefore gives, with
$q^{(1)}:=\mathcal R_i\mathcal R_j(\chi V_iV_j)$,
\begin{equation}
 q_k^{(1)}\longrightarrow q^{(1)}
 \quad\hbox{strongly in }L^{3/2}(\R^5).
 \label{eq:pressure-CZ-strong}
\end{equation}

Taking the divergence of the momentum equation gives
\[
 -\Delta q_k=\partial_i\partial_j(v_{k,i}v_{k,j})
\]
in distributions.  Since $\chi=1$ on $B_0$, the difference
\[
 h_k:=q_k-q_k^{(1)}
\]
is harmonic in $B_0$.  Weak convergence of $q_k$ in $L^{3/2}(B_0)$ and
\eqref{eq:pressure-CZ-strong} imply
\[
 \sup_k\|h_k\|_{L^{3/2}(B_0)}<\infty.
\]
For every integer $m\geq0$, the interior estimates for harmonic functions yield
\begin{equation}
 \|h_k\|_{C^m(\overline B)}
 \leq C_{m,B,B_0}\|h_k\|_{L^{3/2}(B_0)}.
 \label{eq:harmonic-pressure-interior}
\end{equation}
Thus Arzel\`a--Ascoli and a diagonal argument in $m$ give, after extraction,
$h_k\to h$ in $C^m(\overline B)$ for every $m$.  On $B$ we consequently have
\[
 q_k=q_k^{(1)}+h_k\longrightarrow q^{(1)}+h
 \quad\hbox{strongly in }L^{3/2}(B).
\]
The same sequence converges weakly to $Q$ there, so uniqueness of the distributional
limit gives $Q=q^{(1)}+h$ on $B$.  Hence $q_k\to Q$ strongly in $L^{3/2}(B)$.
Applying this construction to a countable locally finite family of balls whose compact
subballs cover $G$, and taking one final diagonal subsequence, proves strong convergence
on every compact subset of $G$.
\end{proof}

\subsection{Head pressure and renormalized structure}

\begin{Prop}
\label{prop:smooth-RHP}
Let $G\subset\R^5$ be open, and let $(U,P)\in C^\infty(G;\R^5\times\R)$ solve \eqref{eq:NS} in $G$. Then $W\in H^1_{\rm loc}(G)$,
\begin{equation}
-\Delta z+\diver(Uz)+|\Omega|^2\1_{\{H>0\}}\leq0
\quad\text{in }\cD'(G),
\label{eq:smooth-Kato-head}
\end{equation}
where $z=H_+$, and
\begin{equation}
\diver\left(\frac43W\nabla W-\frac23UW^2\right)
\geq
\frac89|\nabla W|^2+|\Omega|^2H_+^{1/2}
\quad\text{in }\cD'(G),
\label{eq:smooth-RHP}
\end{equation}
where $W=z^{\frac34}$.
In particular, every smooth stationary Navier--Stokes solution satisfies $\mathrm{(RHP)}$ on its smoothness set.
\end{Prop}

\begin{proof}
We divide the proof into three steps.

\medskip
\noindent
\textbf{Step 1. The head-pressure identity and the positive-part inequality.}
Since $(U,P)$ is smooth, taking the divergence of $\eqref{eq:NS}_1$ gives
\[
-\Delta P=\partial_iU_j\,\partial_jU_i.
\]
On the other hand, taking the scalar product of the momentum equation with $U$ yields
\[
-\Delta\frac{|U|^2}{2}
+\diver(UH)
+|\nabla U|^2=0.
\]
Adding the two identities and using
\[
|\nabla U|^2-\partial_iU_j\,\partial_jU_i
=
\frac12\sum_{i,j=1}^5(\partial_iU_j-\partial_jU_i)^2
=
|\Omega|^2,
\]
we obtain the pointwise identity
\begin{equation}
-\Delta H+\diver(UH)=-|\Omega|^2
\quad\text{in }G.
\label{eq:smooth-head-identity}
\end{equation}

Choose a nondecreasing function $\beta\in C^\infty(\R)$ such that
\[
0\leq\beta\leq1,
\qquad
\beta(s)=0\quad\text{for }s\leq0,
\qquad
\beta(s)=1\quad\text{for }s\geq1,
\]
and define, for $\varepsilon>0$,
\[
\rho_\varepsilon(s)
:=
\int_0^s\beta(t/\varepsilon)\,dt.
\]
Then $\rho_\varepsilon\in C^\infty(\R)$ is nondecreasing and convex,
\[
\rho_\varepsilon'(s)=\beta(s/\varepsilon),
\qquad
\rho_\varepsilon''(s)=\varepsilon^{-1}\beta'(s/\varepsilon)\geq0,
\]
and
\begin{equation}
0\leq\rho_\varepsilon(s)\leq s_+,
\qquad
|\rho_\varepsilon(s)-s_+|\leq\varepsilon,
\qquad
\rho_\varepsilon'(s)\longrightarrow\1_{\{s>0\}}
\quad\text{for every }s\in\R.
\label{eq:positive-part-approximation}
\end{equation}

Since $\diver U=0$, the classical chain rule and \eqref{eq:smooth-head-identity} give
\begin{align}
-\Delta\rho_\varepsilon(H)
+\diver\bigl(U\rho_\varepsilon(H)\bigr)
&=
\rho_\varepsilon'(H)
\bigl[-\Delta H+\diver(UH)\bigr]
-\rho_\varepsilon''(H)|\nabla H|^2
\notag\\
&=
-\rho_\varepsilon'(H)|\Omega|^2
-\rho_\varepsilon''(H)|\nabla H|^2
\notag\\
&\leq
-\rho_\varepsilon'(H)|\Omega|^2.
\label{eq:positive-part-approximate-ineq}
\end{align}

Let $0\leq\phi\in C_c^\infty(G)$.  By \eqref{eq:positive-part-approximation},
\[
\rho_\varepsilon(H)\longrightarrow z
\quad\text{uniformly on }\supp\phi,
\]
and hence
\[
U\rho_\varepsilon(H)\longrightarrow Uz
\quad\text{in }L^1(\supp\phi).
\]
Moreover,
\[
0\leq\rho_\varepsilon'(H)\leq1,
\qquad
\rho_\varepsilon'(H)\longrightarrow\1_{\{H>0\}},
\]
so dominated convergence gives
\[
\rho_\varepsilon'(H)|\Omega|^2
\longrightarrow
|\Omega|^2\1_{\{H>0\}}
\quad\text{in }L^1(\supp\phi).
\]
Passing to the limit $\varepsilon\downarrow0$ in \eqref{eq:positive-part-approximate-ineq} proves \eqref{eq:smooth-Kato-head}.

Since $H$ is smooth, $z=H_+$ belongs to $W^{1,\infty}_{\rm loc}(G)$ and
\begin{equation}
\nabla z=\1_{\{H>0\}}\nabla H
\quad\text{a.e. in }G.
\label{eq:gradient-positive-part}
\end{equation}
Moreover, the left-hand side of \eqref{eq:smooth-Kato-head} is a nonpositive distribution.  Hence there exists a nonnegative Radon measure $\mu$ such that
\begin{equation}
-\Delta z+\diver(Uz)+g=-\mu,
\qquad
g:=|\Omega|^2\1_{\{H>0\}}.
\label{eq:smooth-z-defect}
\end{equation}

\medskip
\noindent
\textbf{Step 2. The square-root approximation and the local $H^1$ bound.}
For $\delta>0$, define
\[
\Theta_\delta(s)
:=
\sqrt{s+\delta}-\sqrt{\delta},
\qquad s\geq0.
\]
Then
\[
\Theta_\delta'(s)=\frac1{2\sqrt{s+\delta}},
\qquad
0\leq\Theta_\delta(s)\leq\sqrt{s}.
\]
Let
\[
B_\delta(s):=\int_0^s\Theta_\delta(t)\,dt
=
\frac23\bigl((s+\delta)^{3/2}-\delta^{3/2}\bigr)-s\sqrt{\delta}.
\]
Then
\begin{equation}
0\leq B_\delta(s)\leq\frac23s^{3/2},
\qquad
B_\delta(s)\longrightarrow\frac23s^{3/2}
\quad\text{as }\delta\downarrow0.
\label{eq:Bdelta-properties}
\end{equation}

Fix $G'\Subset G$ and let $\eta\in C_c^\infty(G')$ be nonnegative.  Since $z\in W^{1,\infty}_{\rm loc}(G)$ and $\Theta_\delta$ is Lipschitz on bounded sets,
\[
\psi_\delta:=\eta^2\Theta_\delta(z)
\]
is a nonnegative compactly supported Lipschitz function.  The test $\psi_\delta$ is admissible in \eqref{eq:smooth-z-defect}: indeed, after extending it by zero outside a compact subset of $G'$, one may mollify it to obtain nonnegative $\psi_{\delta,m}\in C_c^\infty(G')$ such that
\[
\psi_{\delta,m}\longrightarrow\psi_\delta
\quad\text{uniformly and in }W^{1,2}(G').
\]
Since $\nabla z$, $Uz$, and $g$ are locally bounded and $\mu$ is locally finite, all terms in \eqref{eq:smooth-z-defect} pass to the limit $m\to\infty$; hence $\psi_\delta$ is a legitimate test.

Testing \eqref{eq:smooth-z-defect} by $\eta^2\Theta_\delta(z)$ gives
\begin{align}
&\int\eta^2\Theta_\delta'(z)|\nabla z|^2\,dx
+\int g\eta^2\Theta_\delta(z)\,dx
+\int\eta^2\Theta_\delta(z)\,d\mu
\notag\\
&\qquad=
\int B_\delta(z)\Delta(\eta^2)\,dx
+\int B_\delta(z)U\cdot\nabla(\eta^2)\,dx.
\label{eq:square-root-test}
\end{align}
Indeed, the Laplace cross term satisfies
\[
\int\Theta_\delta(z)\nabla z\cdot\nabla(\eta^2)\,dx
=
-\int B_\delta(z)\Delta(\eta^2)\,dx.
\]
For the drift term, define
\[
A_\delta(s):=s\Theta_\delta(s)-B_\delta(s),
\]
so that $A_\delta'(s)=s\Theta_\delta'(s)$.  Using $\diver U=0$,
\begin{align*}
-\int Uz\cdot\nabla\bigl(\eta^2\Theta_\delta(z)\bigr)\,dx
&=
-\int z\Theta_\delta(z)U\cdot\nabla(\eta^2)\,dx
-\int\eta^2U\cdot\nabla A_\delta(z)\,dx\\
&=
\int\bigl[A_\delta(z)-z\Theta_\delta(z)\bigr]U\cdot\nabla(\eta^2)\,dx\\
&=
-\int B_\delta(z)U\cdot\nabla(\eta^2)\,dx,
\end{align*}
which yields \eqref{eq:square-root-test} after moving the cross terms to the right-hand side.

The measure term and the $g$-term in \eqref{eq:square-root-test} are nonnegative.  Hence, using \eqref{eq:Bdelta-properties},
\begin{equation}
\int\eta^2\frac{|\nabla z|^2}{2\sqrt{z+\delta}}\,dx
\leq
C\int z^{3/2}
\bigl(
|\nabla\eta|^2+|\eta\Delta\eta|+|U|\eta|\nabla\eta|
\bigr)\,dx.
\label{eq:Wdelta-bound}
\end{equation}

Define
\[
W_\delta:=(z+\delta)^{3/4}-\delta^{3/4}.
\]
Then
\[
\nabla W_\delta
=
\frac34(z+\delta)^{-1/4}\nabla z,
\]
and therefore
\begin{equation}
\frac12(z+\delta)^{-1/2}|\nabla z|^2
=
\frac89|\nabla W_\delta|^2.
\label{eq:Wdelta-gradient}
\end{equation}
Moreover,
\[
0\leq W_\delta\leq z^{3/4}=W,
\qquad
W_\delta\longrightarrow W
\quad\text{pointwise as }\delta\downarrow0.
\]
Since $W^2=z^{3/2}\in L^1_{\rm loc}(G)$, dominated convergence gives
\[
W_\delta\longrightarrow W
\quad\text{strongly in }L^2_{\rm loc}(G).
\]
On the other hand, \eqref{eq:Wdelta-bound}--\eqref{eq:Wdelta-gradient} give a uniform local $L^2$ bound for $\nabla W_\delta$.  Weak compactness in $H^1$ and the strong $L^2$ convergence therefore imply
\[
W\in H^1_{\rm loc}(G).
\]

\medskip
\noindent
\textbf{Step 3. Passage to the renormalized flux inequality.}
Let $0\leq\phi\in C_c^\infty(G)$.  As above, $\phi\Theta_\delta(z)$ is an admissible nonnegative test in \eqref{eq:smooth-z-defect}.  Repeating the preceding calculation gives
\begin{align}
&\int\phi\Theta_\delta'(z)|\nabla z|^2\,dx
+\int g\phi\Theta_\delta(z)\,dx
+\int\phi\Theta_\delta(z)\,d\mu
\notag\\
&\qquad=
\int B_\delta(z)\Delta\phi\,dx
+\int B_\delta(z)U\cdot\nabla\phi\,dx.
\label{eq:flux-delta}
\end{align}

Discard the nonnegative measure term and let $\delta\downarrow0$.  By the weak lower semicontinuity of the $H^1$ seminorm,
\[
\frac89\int\phi|\nabla W|^2\,dx
\leq
\liminf_{\delta\downarrow0}
\int\phi\Theta_\delta'(z)|\nabla z|^2\,dx.
\]
Since $0\leq\Theta_\delta(z)\leq z^{1/2}$ and $\Theta_\delta(z)\to z^{1/2}$, dominated convergence gives
\[
\int g\phi\Theta_\delta(z)\,dx
\longrightarrow
\int gz^{1/2}\phi\,dx.
\]
Likewise, \eqref{eq:Bdelta-properties} and dominated convergence give
\[
\int B_\delta(z)\bigl(\Delta\phi+U\cdot\nabla\phi\bigr)\,dx
\longrightarrow
\frac23\int z^{3/2}\bigl(\Delta\phi+U\cdot\nabla\phi\bigr)\,dx.
\]
Hence
\begin{equation}
\frac89\int\phi|\nabla W|^2\,dx
+\int gz^{1/2}\phi\,dx
\leq
\frac23\int z^{3/2}\Delta\phi\,dx
+\frac23\int Uz^{3/2}\cdot\nabla\phi\,dx.
\label{eq:flux-limit}
\end{equation}

Since $W^2=z^{3/2}$ and $W\in H^1_{\rm loc}(G)$,
\[
\frac23\int z^{3/2}\Delta\phi\,dx
=
-\frac43\int W\nabla W\cdot\nabla\phi\,dx.
\]
Therefore the right-hand side of \eqref{eq:flux-limit} equals
\[
-\int
\left(
\frac43W\nabla W-\frac23UW^2
\right)\cdot\nabla\phi\,dx.
\]
Since
\[
gz^{1/2}
=
|\Omega|^2\1_{\{H>0\}}H_+^{1/2}
=
|\Omega|^2H_+^{1/2},
\]
we conclude that
\[
-\int
\left(
\frac43W\nabla W-\frac23UW^2
\right)\cdot\nabla\phi\,dx
\geq
\frac89\int\phi|\nabla W|^2\,dx
+
\int\phi|\Omega|^2H_+^{1/2}\,dx.
\]
This is precisely \eqref{eq:smooth-RHP}.  Hence $(U,P)$ satisfies $\mathrm{(RHP)}$ in $G$.
\end{proof}

\begin{Rem}
\label{rem:weak-RHP-gap}
For a general suitable weak solution, the pressure equation and the local energy
inequality imply
\begin{equation}
-\Delta H+\operatorname{div}(UH)+|\Omega|^2=-\mu
\quad\text{in }\mathcal D',
\qquad \mu\geq0,
\label{eq:weak-head-defect}
\end{equation}
where \(H=|U|^2/2+P\) and \(\mu\) is a nonnegative Radon measure.
Nevertheless, \eqref{eq:weak-head-defect} does not automatically imply RHP.
Indeed,
$$
 U\in L^{10/3}_{\rm loc},
 \qquad H\in L^{3/2}_{\rm loc},
 \qquad UH\in L^{30/29}_{\rm loc},
$$
so \(\operatorname{div}(UH)\) need not be a Radon measure.  Consequently,
\(\Delta H\) is not known to be a measure, and the usual Kato positive-part
inequality cannot be applied directly to \(H\).  A mollification argument also
produces an uncontrolled commutator involving
$$
 (UH)_\varepsilon-U_\varepsilon H_\varepsilon.
$$
Thus suitability and Morrey control alone are not claimed to imply RHP.  In this
paper, Proposition~\ref{prop:smooth-RHP} supplies RHP for smooth punctured
solutions, while the reduced RHP inequality for tangent solutions is obtained by
passing to the limit from smooth rescalings.
\end{Rem}

\subsection{Vanishing of the positive head pressure: proof of Proposition \ref{Thm:RHP-sign}}

We now work with a solution $(U,P)$ which is suitable on every compact subset of
$\R^5\setminus\{0\}$, has property $\mathrm{(RHP)}$ there, and satisfies
\begin{equation}
 \sup_{R>0}R^{-2}\int_{B_R}|U|^3\,dx\leq M,
 \qquad
 \sup_{R>0}R^{-2}\int_{B_R}|P|^{3/2}\,dx\leq K.
 \label{eq:globalMorrey}
\end{equation}
It follows that
\begin{equation}
 \sup_{R>0}R^{-2}\int_{B_R}|H|^{3/2}\,dx\leq K_1(M,K).
 \label{eq:HMorrey}
\end{equation}

Let $z=H_+$ and $W=z^{3/4}$.  Testing \eqref{eq:RHP} against a nonnegative cutoff
$\eta^2$ compactly supported away from the origin gives
\begin{align*}
 \frac89\int\eta^2|\nabla W|^2\,dx
 &\leq-\int
 \left(\frac43W\nabla W-\frac23UW^2\right)
 \cdot\nabla(\eta^2)\,dx\\
 &=-\frac83\int\eta W\nabla W\cdot\nabla\eta\,dx
 +\frac43\int\eta W^2U\cdot\nabla\eta\,dx.
\end{align*}
Young's inequality absorbs $4/9$ of the gradient term and yields
\begin{align}
 \frac49\int\eta^2|\nabla W|^2\,dx
 \leq{}&C\int W^2|\nabla\eta|^2\,dx\notag\\
 &+C\int|U|W^2\eta|\nabla\eta|\,dx.
 \label{eq:RHP-Caccioppoli}
\end{align}
This is the only consequence of the Kato--Stampacchia calculation needed below.

\begin{Lem}[Uniform annular head-pressure estimate]
\label{lem:annular-W}
Let
\[
 A_R=B_{2R}\setminus B_{R/2},
 \qquad
 A_R^*=B_{4R}\setminus B_{R/4}.
\]
Then
\begin{equation}
 \int_{A_R}|\nabla W|^2\,dx
 \leq C_{M,K}
 \label{eq:annular-W}
\end{equation}
for every $R>0$.
\end{Lem}

\begin{proof}
For $0\leq t\leq1$, put
\[
 \mathcal A(t):=B_{(2+2t)R}\setminus
 B_{R/(2+2t)}.
\]
Thus $\mathcal A(0)=A_R$ and $\mathcal A(1)=A_R^*$.  If $0\leq s<t\leq1$, choose
$\eta$ which is one on $\mathcal A(s)$, supported in $\mathcal A(t)$, and satisfies
\[
 |\nabla\eta|\leq \frac{C}{R(t-s)},
 \qquad
 |\Delta\eta|\leq \frac{C}{R^2(t-s)^2}.
\]
Write
\[
 \Phi(t):=\int_{\mathcal A(t)}|\nabla W|^2\,dx,
 \qquad
 L:=\int_{A_R^*}W^2\,dx.
\]
Estimate \eqref{eq:RHP-Caccioppoli} then gives
\begin{equation}
 \Phi(s)\leq
 \frac{C}{R^2(t-s)^2}L
+\frac{C}{R(t-s)}\int_{\mathcal A(t)}|U|W^2\,dx.
 \label{eq:Phi-pre}
\end{equation}

On every $\mathcal A(t)$, H\"older and the five-dimensional
Gagliardo--Nirenberg inequality give
\begin{align}
 \int_{\mathcal A(t)}|U|W^2\,dx
 &\leq \|U\|_{L^3(\mathcal A(t))}
 \|W\|_{L^3(\mathcal A(t))}^2,
 \label{eq:drift1}\\
 \|W\|_{L^3(\mathcal A(t))}^2
 &\leq C\|W\|_{L^2(\mathcal A(t))}^{1/3}
 \left(\|\nabla W\|_{L^2(\mathcal A(t))}
 +R^{-1}\|W\|_{L^2(\mathcal A(t))}\right)^{5/3}.
 \label{eq:GN-W}
\end{align}
The constants are uniform in $t$, because these annuli have uniformly Lipschitz
rescalings, and the Morrey bound gives
\[
 \|U\|_{L^3(\mathcal A(t))}\leq C M^{1/3}R^{2/3}.
\]
Young's inequality with conjugate exponents $6/5$ and $6$ therefore gives, for every
$0<\theta<1$,
\begin{equation}
 \frac{C}{R(t-s)}\int_{\mathcal A(t)}|U|W^2\,dx
 \leq \theta\Phi(t)
 +\frac{C_{\theta,M}}{R^2(t-s)^6}L.
 \label{eq:drift-absorb}
\end{equation}
Inserting this into \eqref{eq:Phi-pre} and enlarging the constant yields
\begin{equation}
 \Phi(s)\leq\theta\Phi(t)
 +\frac{C_{\theta,M}}{R^2(t-s)^6}L.
 \label{eq:hole-recursion}
\end{equation}
For completeness, take $t_j=1-2^{-j}$ and apply
\eqref{eq:hole-recursion} with $(s,t)=(t_j,t_{j+1})$.  Iteration gives
\[
 \Phi(0)\leq \theta^N\Phi(t_N)
 +\frac{C_{\theta,M}L}{R^2}
 \sum_{j=0}^{N-1}\theta^j2^{6(j+1)}.
\]
Choose $\theta<2^{-6}$.  Since $\Phi(t_N)\leq\Phi(1)<\infty$, letting
$N\to\infty$ proves
\[
 \int_{A_R}|\nabla W|^2
 \leq C_MR^{-2}\int_{A_R^*}W^2.
\]
Finally $W^2=H_+^{3/2}$, so \eqref{eq:HMorrey} proves
\eqref{eq:annular-W}.
\end{proof}

Define the renormalized flux
\begin{equation}
 \cF:=H_+^{1/2}\nabla H_+-\frac23UH_+^{3/2}
 =\frac43W\nabla W-\frac23UW^2.
 \label{eq:F}
\end{equation}
Property $\mathrm{(RHP)}$ gives
\begin{equation}
 \diver\cF\geq\frac89|\nabla W|^2
 \quad\text{in }\cD'(\R^5\setminus\{0\}).
 \label{eq:divF}
\end{equation}

\begin{proof}[Proof of Proposition \ref{Thm:RHP-sign}]
First, $\cF\in L^1_{\rm loc}$.  Indeed, using \eqref{eq:HMorrey},
Lemma \ref{lem:annular-W}, and \eqref{eq:GN-W},
\begin{align*}
 \int_{A_R}W|\nabla W|\,dx&\leq C_{M,K}R,\\
 \int_{A_R}|U|W^2\,dx&\leq C_{M,K}R.
\end{align*}
Thus
\begin{equation}
 \int_{A_R}|\cF|\,dx\leq C_{M,K}R.
 \label{eq:Fannular}
\end{equation}

For almost every $R>0$, the coarea formula defines
\begin{equation}
 \Xi_1(R):=\int_{\partial B_R}\cF\cdot e_r\,dS.
 \label{eq:g2}
\end{equation}
No pointwise normal trace is needed: for a radial test $\psi(|x|)$,
\[
 -\int\cF\cdot\nabla\psi\,dx
 =-\int_0^\infty\psi'(R)\Xi_1(R)\,dR.
\]
Consequently, \eqref{eq:divF} says that $\Xi_1$ has a nondecreasing representative and,
for almost every $0<R_1<R_2$,
\begin{equation}
 \Xi_1(R_2)-\Xi_1(R_1)
 \geq\frac89\int_{B_{R_2}\setminus B_{R_1}}|\nabla W|^2\,dx.
 \label{eq:g2mono}
\end{equation}
By \eqref{eq:Fannular},
\[
 \frac1R\int_R^{2R}|\Xi_1(s)|\,ds\leq C_{M,K}.
\]
Hence $\Xi_1$ has bounded subsequences at both zero and infinity.  Its monotonicity then
implies that both endpoint limits are finite.  Letting the endpoints tend to zero and
infinity in \eqref{eq:g2mono} yields
\begin{equation}
 \int_{\R^5}|\nabla W|^2\,dx<\infty.
 \label{eq:global-W-energy}
\end{equation}

Since $W\in L^2_{\rm loc}(\R^5)$ by \eqref{eq:HMorrey}, and a point has zero
$H^1$-capacity in dimension five, \eqref{eq:global-W-energy} extends $W$ across the
origin as an element of $H^1_{\rm loc}(\R^5)$.  The homogeneous Sobolev inequality
therefore gives a constant $c$ such that
$W-c\in L^{10/3}(\R^5)$.  On the other hand, \eqref{eq:HMorrey} gives
\[
 |(W)_{B_R}|\leq
 \left( \frac{1}{|B_R|}\int_{B_R}W^2\,dx\right)^{1/2}
 \leq C_{M,K}R^{-3/2}\to0,
\]
so $c=0$ and
\begin{equation}
 W\in L^{10/3}(\R^5).
 \label{eq:W103}
\end{equation}

We now improve the averaged flux bound.  H\"older's inequality gives
\begin{align*}
 \frac1R\int_{A_R}W|\nabla W|\,dx
 &\leq
 C\|W\|_{L^{10/3}(A_R)}\|\nabla W\|_{L^2(A_R)},\\
 \frac1R\int_{A_R}|U|W^2\,dx
 &\leq
 CR^{-1}\|U\|_{L^{5/2}(A_R)}
 \|W\|_{L^{10/3}(A_R)}^2\\
 &\leq CM^{1/3}\|W\|_{L^{10/3}(A_R)}^2.
\end{align*}
Therefore
\begin{equation}
 \frac1R\int_R^{2R}|\Xi_1(s)|\,ds
 \leq
 C\|W\|_{L^{10/3}(A_R)}\|\nabla W\|_{L^2(A_R)}
 +CM^{1/3}\|W\|_{L^{10/3}(A_R)}^2.
 \label{eq:g2vanish-average}
\end{equation}
The right side tends to zero as $R\downarrow0$ and as $R\uparrow\infty$, by
\eqref{eq:global-W-energy}--\eqref{eq:W103}.  Thus there are sequences at both ends
along which $\Xi_1$ tends to zero.  Since $\Xi_1$ is nondecreasing, both endpoint limits are
zero and $\Xi_1\equiv0$.  Equation \eqref{eq:g2mono} gives $\nabla W=0$.  Finally,
\eqref{eq:W103} forces $W=0$.  Hence
\begin{equation}
 H\leq0\quad\text{a.e. in }\R^5,
 \label{eq:Hnegative}
\end{equation}
which proves Proposition \ref{Thm:RHP-sign}.
\end{proof}


\section{The weak sign-rigidity: proof of Proposition \ref{thm:sign-rigidity}}

In this section the sign information $H\leq0$ is converted into full rigidity.  A
radial distributional test yields the exact monotonicity formula for $\Xi_2$; an
independent zero-defect argument classifies every possible tangent with vanishing
weighted defect.  Finally, blow-up at the origin and blow-down at infinity identify
both endpoint fluxes as zero.

\subsection{The radial monotonicity identity}

Write $U_r=U\cdot e_r$, where $e_r=x/|x|$.  For almost every $R>0$, define
\begin{equation}
 \Xi_2(R):=R^{-2}\int_{\partial B_R}(U_r^2+P)\,dS.
 \label{eq:g3}
\end{equation}

\begin{Lem}
\label{lem:g3}
Let $(U,P)$ be a global suitable weak solution to \eqref{eq:NS} satisfying \eqref{eq:globalMorrey}.
Then, for almost every $0<R_1<R_2$,
\begin{equation}
 \Xi_2(R_1)-\Xi_2(R_2)
 =\int_{B_{R_2}\setminus B_{R_1}}
 \frac{3U_r^2-2H}{|x|^3}\,dx.
 \label{eq:g3identity}
\end{equation}
In particular, if $H\leq0$, then $\Xi_2$ is nonincreasing and
\begin{equation}
 \int_{\R^5}\frac{3U_r^2-2H}{|x|^3}\,dx<\infty.
 \label{eq:weighted-defect}
\end{equation}
\end{Lem}

\begin{proof}
We give the approximation argument in detail, because the values of a weak solution
on a prescribed sphere need not be defined.  Since
$U\otimes U,P\in L^1_{\rm loc}(\R^5)$, the coarea formula shows that the functions
\begin{align*}
 G(r)&:=\int_{\partial B_r}(U_r^2+P)\,dS,\\
 J(r)&:=\int_{\partial B_r}(|U|^2-U_r^2+4P)\,dS
\end{align*}
belong to $L^1_{\rm loc}(0,\infty)$.  In what follows, $R_1<R_2$ are chosen to be
Lebesgue points of $G$.  This excludes only a null set of radii and is precisely the
meaning of the surface integrals occurring in \eqref{eq:g3}.

Let $\rho\in C_c^\infty((-1,1))$ be nonnegative, even, and satisfy
$\int_{\R}\rho=1$, and put $\rho_\varepsilon(r)=\varepsilon^{-1}
\rho(r/\varepsilon)$.  Extend the function
\begin{equation}
 q(r):=-r^{-2}\mathbf 1_{(R_1,R_2)}(r)
 \label{eq:q-radial}
\end{equation}
by zero to the whole real line, and set
\[
 q_\varepsilon=\rho_\varepsilon*q,
 \qquad
 f_\varepsilon(r):=-\int_r^\infty q_\varepsilon(s)\,ds,
 \qquad
 \psi_\varepsilon(x):=f_\varepsilon(|x|).
\]
For $0<\varepsilon<R_1/2$, the function $f_\varepsilon$ is constant near zero and
vanishes for $r\geq R_2+\varepsilon$.  Hence
$\psi_\varepsilon\in C_c^\infty(\R^5)$, and
\[
 f_\varepsilon'=q_\varepsilon,
 \qquad f_\varepsilon''=q_\varepsilon'.
\]
The limiting function differs by an irrelevant additive constant from the continuous
radial function that equals $R_1^{-1}$ in $B_{R_1}$, $|x|^{-1}$ in the annulus, and
$R_2^{-1}$ outside $B_{R_2}$.

We may test the distributional momentum equation with
$\nabla\psi_\varepsilon$.  The viscous term is exactly zero, because
\begin{equation}
 \langle-\Delta U,\nabla\psi_\varepsilon\rangle
 =\langle\diver U,\Delta\psi_\varepsilon\rangle=0.
 \label{eq:viscous-gradient-zero}
\end{equation}
After multiplying the remaining equality by $-1$, we obtain
\begin{equation}
 \int_{\R^5}\left(
 U_iU_j\partial_{ij}\psi_\varepsilon
 +P\Delta\psi_\varepsilon\right)\,dx=0.
 \label{eq:radial-test-stress}
\end{equation}
For an arbitrary smooth radial function $f=f(r)$ in five dimensions,
\begin{align}
 \partial_{ij}f
 &=f''e_{r,i}e_{r,j}
   +\frac{f'}r(\delta_{ij}-e_{r,i}e_{r,j}),
 \notag\\
 \Delta f&=f''+4\frac{f'}r.
 \label{eq:radial-hessian-general}
\end{align}
Consequently, the coarea formula turns \eqref{eq:radial-test-stress} into the
one-dimensional identity
\begin{equation}
 0=\int_0^\infty G(r)q_\varepsilon'(r)\,dr
   +\int_0^\infty J(r)\frac{q_\varepsilon(r)}r\,dr.
 \label{eq:radial-test-one-dimensional}
\end{equation}

We now compute the limit of both terms.  The distributional derivative of
\eqref{eq:q-radial} is
\begin{equation}
 q'
 =2r^{-3}\mathbf 1_{(R_1,R_2)}(r)
  -R_1^{-2}\delta_{R_1}
  +R_2^{-2}\delta_{R_2}.
 \label{eq:q-prime-atoms}
\end{equation}
Indeed, the classical derivative in $(R_1,R_2)$ is $2r^{-3}$, while the jumps
$q(R_1+)-q(R_1-)=-R_1^{-2}$ and
$q(R_2+)-q(R_2-)=R_2^{-2}$ give the two atoms.  Extend $G$ by zero to the negative
axis.  Since $\rho$ is even,
\[
 \int G(r)(\rho_\varepsilon*q')(r)\,dr
 =\bigl\langle q',\rho_\varepsilon*G\bigr\rangle.
\]
The approximate-identity theorem gives
$\rho_\varepsilon*G\to G$ in $L^1$ on compact subintervals of $(0,\infty)$; because
$r^{-3}$ is bounded near $[R_1,R_2]$, this treats the absolutely continuous part of
$q'$.  At the atoms, the Lebesgue-point property gives
$(\rho_\varepsilon*G)(R_i)\to G(R_i)$.  Therefore
\begin{align}
 \lim_{\varepsilon\downarrow0}
 \int_0^\infty G(r)q_\varepsilon'(r)\,dr
 ={}&2\int_{R_1}^{R_2}\frac{G(r)}{r^3}\,dr
   -R_1^{-2}G(R_1)+R_2^{-2}G(R_2).
 \label{eq:qprime-limit}
\end{align}
Here the endpoint terms can equivalently be read as the pairing of the singular
Hessian
\[
 -R_1^{-2}e_r\otimes e_r\,\mathcal H^4\lfloor\partial B_{R_1}
 +R_2^{-2}e_r\otimes e_r\,\mathcal H^4\lfloor\partial B_{R_2}
\]
and its trace with $U\otimes U+PI$.  Formula \eqref{eq:qprime-limit}, rather than a
pointwise normal trace, is the rigorous definition of that pairing for the present
weak solution.

Moreover, $q_\varepsilon\to q$ almost everywhere, the functions $q_\varepsilon$ are
uniformly bounded on a fixed compact subinterval of $(0,\infty)$, and $J$ is locally
integrable.  Therefore
\begin{equation}
 \lim_{\varepsilon\downarrow0}
 \int_0^\infty J(r)\frac{q_\varepsilon(r)}r\,dr
 =-\int_{R_1}^{R_2}\frac{J(r)}{r^3}\,dr.
 \label{eq:q-limit}
\end{equation}
Combining \eqref{eq:radial-test-one-dimensional}--\eqref{eq:q-limit} gives
\begin{align*}
 0={}&\int_{R_1}^{R_2}\frac{2G(r)-J(r)}{r^3}\,dr
      -R_1^{-2}G(R_1)+R_2^{-2}G(R_2)\\
 ={}&\int_{B_{R_2}\setminus B_{R_1}}
       \frac{3U_r^2-|U|^2-2P}{|x|^3}\,dx
      -\Xi_2(R_1)+\Xi_2(R_2)\\
 ={}&\int_{B_{R_2}\setminus B_{R_1}}
       \frac{3U_r^2-2H}{|x|^3}\,dx
      -\Xi_2(R_1)+\Xi_2(R_2).
\end{align*}
This proves \eqref{eq:g3identity}.  Notice in particular that the signs in the last
two terms come directly from the two jumps in \eqref{eq:q-prime-atoms}.  The argument
applies whenever both endpoints are Lebesgue points of $G$, and hence for almost every
pair $0<R_1<R_2$.

Assume now that $H\leq0$.  It remains to prove finiteness.  By H\"older and
\eqref{eq:globalMorrey},
\[
 \int_{A_R}|U|^2\,dx\leq C M^{2/3}R^3,
 \qquad
 \int_{A_R}|P|\,dx\leq C K^{2/3}R^3.
\]
Consequently,
\begin{equation}
 \frac1R\int_R^{2R}|\Xi_2(s)|\,ds
 \leq \frac{C}{R^3}\int_{A_R}(|U|^2+|P|)\,dx
 \leq C_{M,K}.
 \label{eq:g3average}
\end{equation}
To make the monotonicity statement precise, define
\[
 d(r):=r^{-3}\int_{\partial B_r}(3U_r^2-2H)\,dS.
\]
Then $d\in L^1_{\rm loc}(0,\infty)$, $d\geq0$, and
\eqref{eq:g3identity} says that, for almost every $R_1<R_2$,
$\Xi_2(R_1)-\Xi_2(R_2)=\int_{R_1}^{R_2}d(r)\,dr$.  Fix a Lebesgue point
$r_0$ of $G$ and define
\[
 \bar \Xi_2(r):=
 \begin{cases}
  \Xi_2(r_0)+\displaystyle\int_r^{r_0}d(s)\,ds,&0<r<r_0,\\[6pt]
  \Xi_2(r_0)-\displaystyle\int_{r_0}^{r}d(s)\,ds,&r\geq r_0.
 \end{cases}
\]
The identity already proved for every pair of Lebesgue points of $G$ shows that
$\bar \Xi_2=\Xi_2$ almost everywhere.  Thus $\bar \Xi_2$ is an absolutely continuous,
nonincreasing representative on every compact subinterval of $(0,\infty)$.

The uniform dyadic average bound \eqref{eq:g3average} supplies, for every $R>0$, a
Lebesgue point $s_R\in(R,2R)$ at which $|\bar \Xi_2(s_R)|\leq C_{M,K}$; otherwise the
average in \eqref{eq:g3average} would be larger than $C_{M,K}$.  Taking
$R\downarrow0$ and $R\uparrow\infty$ and using monotonicity shows that the endpoint limits
$\ell_0=\lim_{r\downarrow0}\bar \Xi_2(r)$ and
$\ell_\infty=\lim_{r\uparrow\infty}\bar \Xi_2(r)$ are both finite.  Finally, letting
$R_1\downarrow0$ and $R_2\uparrow\infty$ in the nonnegative identity and applying
monotone convergence gives
\begin{align}
    \int_{\R^5}\frac{3U_r^2-2H}{|x|^3}\,dx
 =\ell_0-\ell_\infty<\infty, \label{eq:weighted-defect2}
\end{align}
which is \eqref{eq:weighted-defect}.
\end{proof}

\subsection{Zero-defect tangents}

The next proposition is the second key point of the proof.

\begin{Prop}[Zero-defect tangent is trivial]
\label{prop:zero-defect}
Let $(V_k,Q_k)$ be suitable weak solutions to \eqref{eq:NS} on $\R^5\setminus\{0\}$, and write
\[
 H_k=\frac{|V_k|^2}{2}+Q_k.
\]
Assume, on every compact annulus $A\Subset\R^5\setminus\{0\}$, that
\begin{align}
 V_k&\to V &&\text{strongly in }L^3(A),\label{eq:Vkstrong}\\
 V_k&\rightharpoonup V &&\text{weakly in }H^1(A),\label{eq:VkweakH1}\\
 Q_k&\rightharpoonup Q &&\text{weakly in }L^{3/2}(A),\label{eq:Qkweak}\\
 H_k&\leq0,\label{eq:Hknegative}
\end{align}
and
\begin{equation}
 \int_A\left(3|(V_k)_r|^2-2H_k\right)\,dx\to0.
 \label{eq:defect-zero}
\end{equation}
Then $V\equiv0$ and $Q\equiv0$ on $\R^5\setminus\{0\}$.
\end{Prop}

\begin{proof}
Because both terms in \eqref{eq:defect-zero} are nonnegative,
\begin{equation}
 (V_k)_r\to0\quad\text{in }L^2(A),
 \qquad H_k\to0\quad\text{in }L^1(A).
 \label{eq:radial-H-conv}
\end{equation}
Strong $L^3$ convergence implies
\[
 |V_k|^2\to|V|^2\quad\text{strongly in }L^{3/2}(A).
\]
Since $Q_k=H_k-|V_k|^2/2$, we have
\[
 Q_k\to-\frac{|V|^2}{2}\quad\text{in }L^1(A).
\]
Comparison with \eqref{eq:Qkweak} gives
\begin{equation}
 Q=-\frac{|V|^2}{2},
 \qquad H_V:=\frac{|V|^2}{2}+Q=0.
 \label{eq:HVzero}
\end{equation}
Moreover, \eqref{eq:Vkstrong} and \eqref{eq:radial-H-conv} imply
\begin{equation}
 V_r=0.
 \label{eq:Vrzero}
\end{equation}

It remains to prove rigorously that $H_V=0$ forces the antisymmetric gradient to vanish.
Suitability passes to the limit, so for every nonnegative $\phi\in C_c^\infty(A)$,
\begin{equation}
 \int |\nabla V|^2\phi\,dx
 \leq\int\frac{|V|^2}{2}\Delta\phi\,dx,
 \label{eq:LEI-Hzero}
\end{equation}
where \eqref{eq:HVzero} was used.  Taking the divergence of the limiting momentum
equation gives
\[
 -\Delta Q=\partial_iV_j\partial_jV_i.
\]
As $Q=-|V|^2/2$,
\begin{equation}
 \Delta\frac{|V|^2}{2}=\partial_iV_j\partial_jV_i
 \quad\text{in }\cD'(A).
 \label{eq:e-Poisson}
\end{equation}
Using \eqref{eq:e-Poisson} in \eqref{eq:LEI-Hzero},
\[
 \int\left(|\nabla V|^2-\partial_iV_j\partial_jV_i\right)\phi\,dx\leq0.
\]
The integrand equals
\[
 \frac12\sum_{i,j}(\partial_iV_j-\partial_jV_i)^2\geq0.
\]
Hence
\begin{equation}
 \partial_iV_j-\partial_jV_i=0
 \quad\text{a.e. in }\R^5\setminus\{0\}.
 \label{eq:curlzero}
\end{equation}

Together with $\diver V=0$, equation \eqref{eq:curlzero} implies $\Delta V=0$ in
distributions, so $V$ is smooth.  Since $\R^5\setminus\{0\}$ is simply connected,
$V=\nabla\Phi$ for a harmonic scalar function $\Phi$.  Equation \eqref{eq:Vrzero}
gives
\[
 x\cdot\nabla\Phi=0.
\]
Thus $\Phi$ is homogeneous of degree zero, $\Phi(x)=\varphi(x/|x|)$.  Since
\[
 0=\Delta\Phi=|x|^{-2}\Delta_{\mathbb S^4}\varphi,
\]
$\varphi$ is a harmonic function on the compact sphere and is therefore constant.
Consequently $V=0$, and \eqref{eq:HVzero} also gives $Q=0$.
\end{proof}

\subsection{Completion of the sign-rigidity proof}

\begin{proof}[Proof of Proposition \ref{thm:sign-rigidity}]
Set $H=|U|^2/2+P$.  By the sign assumption \eqref{eq:sign-assumption} and
Lemma \ref{lem:g3}, the function
\[
 \Xi_2(R)=R^{-2}\int_{\partial B_R}(U_r^2+P)\,dS
\]
has a nonincreasing representative, its endpoint limits
\begin{equation}
 \ell_0:=\lim_{R\downarrow0}\Xi_2(R),
 \qquad
 \ell_\infty:=\lim_{R\uparrow\infty}\Xi_2(R)
 \label{eq:g3-endpoints}
\end{equation}
are finite, and
\begin{equation}
 \mathcal D(U,P):=
 \int_{\R^5}\frac{3U_r^2-2H}{|x|^3}\,dx<\infty.
 \label{eq:global-defect-again}
\end{equation}
We prove that both limits in \eqref{eq:g3-endpoints} are zero.

Let first $\lambda_k\downarrow0$ and define
\begin{equation}
 U_k(x)=\lambda_kU(\lambda_kx),
 \qquad P_k(x)=\lambda_k^2P(\lambda_kx),
 \qquad H_k(x)=\lambda_k^2H(\lambda_kx).
 \label{eq:global-rescaling}
\end{equation}
The global Morrey bounds are invariant.  Lemma \ref{lem:compactness} therefore gives,
after extraction, a global suitable weak limit $(V,Q)$ with
\begin{equation}
 U_k\to V\text{ strongly in }L^3_{\rm loc},
 \qquad P_k\rightharpoonup Q\text{ weakly in }L^{3/2}_{\rm loc}.
 \label{eq:endpoint-compactness}
\end{equation}
For each compact annulus $A\Subset\R^5\setminus\{0\}$, scale invariance gives
\begin{equation}
 \int_A\frac{3|(U_k)_r|^2-2H_k}{|x|^3}\,dx
 =\int_{\lambda_kA}\frac{3U_r^2-2H}{|y|^3}\,dy\longrightarrow0.
 \label{eq:small-end-defect}
\end{equation}
The convergence follows from the absolute continuity of the integrable function in
\eqref{eq:global-defect-again}, since $\lambda_kA$ shrinks to the origin.  Since
$|x|^{-3}$ is bounded above and below on $A$, the corresponding unweighted defect also
tends to zero.  Proposition \ref{prop:zero-defect} shows that $V=Q=0$ on
$\R^5\setminus\{0\}$.

Denote by $g_{3,k}$ the radial flux associated with $(U_k,P_k)$.  A change of variables
on spheres gives the exact scaling law
\begin{equation}
 g_{3,k}(r)=\Xi_2(\lambda_kr).
 \label{eq:g3-scaling}
\end{equation}
Fix a nonnegative $\psi\in C_c^\infty((1,2))$ with
$\int_1^2\psi(r)\,dr>0$.  The coarea formula and
\eqref{eq:endpoint-compactness} yield
\begin{align}
 \int_1^2\psi(r)\Xi_2(\lambda_kr)\,dr
 &=\int_{B_2\setminus B_1}
 \psi(|x|)|x|^{-2}\bigl((U_k)_r^2+P_k\bigr)\,dx
 \longrightarrow0.
 \label{eq:g3-average-limit}
\end{align}
Indeed, $(U_k)_r^2\to0$ strongly in $L^{3/2}$ on the annulus, while $P_k\rightharpoonup0$
in $L^{3/2}$ and the fixed weight belongs to $L^3$.  On the other hand, recalling \eqref{eq:g3-endpoints}, we have
$\Xi_2(\lambda_kr)\to\ell_0$ for every $r\in(1,2)$.  Since the monotone function $\Xi_2$
is bounded between its two finite endpoint limits, dominated convergence in
\eqref{eq:g3-average-limit} gives
\[
 \ell_0\int_1^2\psi(r)\,dr=0.
\]
Thus $\ell_0=0$.

Now let $\lambda_k\uparrow\infty$ and repeat the same argument.  Compactness again gives
a limit $(V,Q)$, while
\[
 \int_{\lambda_kA}\frac{3U_r^2-2H}{|y|^3}\,dy\longrightarrow0
\]
because $\lambda_kA$ escapes to infinity and the defect is integrable.  Proposition
\ref{prop:zero-defect} again applies, using the boundedness of $|x|^{\pm3}$ on $A$, and
gives $V=Q=0$.  Equations
\eqref{eq:g3-scaling}--\eqref{eq:g3-average-limit}, now followed by dominated convergence
as $\lambda_k\to\infty$, prove $\ell_\infty=0$.

The function $\Xi_2$ is nonincreasing and has both endpoint limits equal to zero; hence
$\Xi_2\equiv0$.  The identity \eqref{eq:g3identity} and the nonnegativity of its integrand
give
\begin{equation}
 U_r=0,
 \qquad H=0
 \quad\text{a.e. in }\R^5\setminus\{0\}.
 \label{eq:global-zero-defect}
\end{equation}
Apply Proposition \ref{prop:zero-defect} to the constant sequence
$(V_k,Q_k)=(U,P)$ on every compact annulus.  Condition
\eqref{eq:global-zero-defect} makes its defect identically zero, so the proposition gives
$U=P=0$ on $\R^5\setminus\{0\}$.  Since a single point has measure zero, the same holds
in $\R^5$.  This proves Proposition \ref{thm:sign-rigidity}.
\end{proof}

\section{The five-dimensional velocity-only rigidity: proof of Theorem \ref{thm:smooth-punctured}}

In this section, we combine the preceding weak mechanisms to prove Theorem \ref{thm:smooth-punctured}.  The canonical pressure is first reconstructed from the
velocity, the punctured smooth solution is then extended across the origin as a
global suitable weak solution, and punctured smoothness supplies RHP.  The implications
$\mathrm{RHP}\Rightarrow H\leq0$ and $H\leq0\Rightarrow U=P=0$ complete the proof.
Thus neither a pressure Morrey bound nor RHP is assumed in the theorem itself.

\begin{proof}[Proof of Theorem \ref{thm:smooth-punctured}]
Apply Proposition \ref{prop:auto-pressure} to the punctured solution $(U,P_0)$.
There is a unique constant $c$ such that
\[
 P:=P^\sharp=P_0-c
\]
satisfies the pressure Morrey estimate \eqref{eq:intro-canonical-pressure}, and
$(U,P)$ solves the same smooth punctured equations.  Thus
$\mathfrak M(U,P)<\infty$.  Notice that this preliminary step uses only the velocity
Morrey bound; no estimate for the originally chosen representative $P_0$ is assumed.

We first show that the punctured solution extends across the origin as a global suitable
weak solution.  Let $\chi_\varepsilon$ be a radial cutoff satisfying
\[
 \chi_\varepsilon=0\text{ on }B_\varepsilon,
 \qquad
 \chi_\varepsilon=1\text{ outside }B_{2\varepsilon},
 \qquad
 |\nabla\chi_\varepsilon|\leq C\varepsilon^{-1},
 \quad
 |\Delta\chi_\varepsilon|\leq C\varepsilon^{-2}.
\]
Fix $R>0$ and choose $\eta\in C_c^\infty(B_{2R})$ with $\eta=1$ on $B_R$.
The smooth local energy equality on $\R^5\setminus\{0\}$, tested with
$\eta^2\chi_\varepsilon$, gives
\begin{align}
 \int_{B_R\setminus B_{2\varepsilon}}|\nabla U|^2\,dx
 \leq C_R
 &+C\varepsilon^{-2}\int_{A_\varepsilon}|U|^2\,dx\notag\\
 &+C\varepsilon^{-1}\int_{A_\varepsilon}
 \bigl(|U|^3+|P||U|\bigr)\,dx,
 \label{eq:punctured-Caccioppoli}
\end{align}
where $A_\varepsilon=B_{2\varepsilon}\setminus B_\varepsilon$, and $C_R$ is
independent of $\varepsilon$.  The Morrey bounds imply
\begin{align}
 \int_{A_\varepsilon}|U|^2\,dx
 &\leq
 \left(\int_{B_{2\varepsilon}}|U|^3\,dx\right)^{2/3}
 |B_{2\varepsilon}|^{1/3}
 \leq C_M\varepsilon^3,
 \label{eq:punctured-U2}\\
 \int_{A_\varepsilon}|P||U|\,dx
 &\leq
 \left(\int_{B_{2\varepsilon}}|P|^{3/2}\,dx\right)^{2/3}
 \left(\int_{B_{2\varepsilon}}|U|^3\,dx\right)^{1/3}
 \leq C_{M,K}\varepsilon^2.
 \label{eq:punctured-PU}
\end{align}
Also $\int_{A_\varepsilon}|U|^3\,dx\leq C_M\varepsilon^2$.
Consequently all inner-cutoff terms in \eqref{eq:punctured-Caccioppoli} are
$O(\varepsilon)$, and monotone convergence gives
\begin{equation}
 U\in H^1_{\rm loc}(\R^5).
 \label{eq:punctured-H1-extension}
\end{equation}

We next verify the equations across the origin.  Test the weak momentum equation on the
punctured space with $\varphi\chi_\varepsilon$, where
$\varphi\in C_c^\infty(\R^5;\R^5)$.  The only new viscous error is bounded by
\[
 \|\nabla U\|_{L^2(A_\varepsilon)}
 \|\varphi\nabla\chi_\varepsilon\|_{L^2(A_\varepsilon)}
 \leq C_\varphi\varepsilon^{3/2}
 \|\nabla U\|_{L^2(A_\varepsilon)}\longrightarrow0.
\]
Using \eqref{eq:punctured-U2} and H\"older's inequality, the convection and pressure
errors satisfy
\[
 \varepsilon^{-1}\int_{A_\varepsilon}|U|^2\,dx=O(\varepsilon^2),
 \qquad
 \varepsilon^{-1}\int_{A_\varepsilon}|P|\,dx=O(\varepsilon^2).
\]
The divergence equation is treated in the same way, since
$\varepsilon^{-1}\int_{A_\varepsilon}|U|\,dx=O(\varepsilon^3)$.
Letting $\varepsilon\downarrow0$ proves \eqref{eq:NS} in $\cD'(\R^5)$.

Finally, insert $\phi\chi_\varepsilon$ into the smooth local energy equality for a
fixed nonnegative $\phi\in C_c^\infty(\R^5)$.  The terms containing derivatives of
$\chi_\varepsilon$ are bounded by exactly the quantities in
\eqref{eq:punctured-Caccioppoli} and hence tend to zero.  The other terms converge by
\eqref{eq:punctured-H1-extension}, the velocity Morrey bound, and the pressure-velocity
estimate \eqref{eq:punctured-PU}.  Thus the extension satisfies \eqref{eq:LEI} and is a
global suitable weak solution.

Proposition \ref{prop:smooth-RHP} shows that the original smooth solution has property
$\mathrm{(RHP)}$ on $\R^5\setminus\{0\}$.  Proposition \ref{Thm:RHP-sign} therefore
gives $H\leq0$ a.e. in $\R^5$.  Applying Proposition
\ref{thm:sign-rigidity} to the suitable extension yields $U=P=0$.  Since
$P_0=P+c$, the original pressure is the constant $c$. The proof is complete.
\end{proof}

\section{The five-dimensional regularity criteria: proof of Theorems \ref{thm:dirichlet-removability} and \ref{cor:first-singularity}}

This section returns from global rigidity to local regularity.  The first subsection
upgrades either bounded Dirichlet scale or bounded velocity scale to simultaneous
$C_5$ and $D_5$ control, using dyadic control of the velocity mean and contraction of
the harmonic pressure.  The second subsection performs the singular-point blow-up,
passes the reduced RHP inequality through strong annular pressure convergence, and
contradicts the nontriviality supplied by one-scale $\varepsilon$-regularity.

\subsection{From the Dirichlet or velocity scale to the pressure Morrey scale}

The next lemma supplies the iteration needed for Theorem
\ref{thm:dirichlet-removability}.  It is important to estimate the velocity mean before
using Sobolev--Poincar\'e, and to estimate the pressure through its harmonic part rather
than by a whole-space pressure formula.

\begin{Lem}
\label{lem:Dirichlet-to-Morrey}
Let $(u,p)$ be a suitable weak solution of \eqref{eq:NS} in $B_1\subset\R^5$.
Suppose that for some $R_0\in(0,1)$,
\begin{equation}
 E_*:=\sup_{0<r<R_0}r^{-1}\int_{B_r}|\nabla u|^2\,dx<\infty.
 \label{eq:E-star}
\end{equation}
Set $R_*:=R_0/4$.  Then
\begin{equation}
 \sup_{0<r<R_*}C_5(r)+\sup_{0<r<R_*}D_5(r)<\infty.
 \label{eq:Dirichlet-Morrey-conclusion}
\end{equation}
More precisely,
\begin{align}
 \sup_{0<r<R_*}C_5(r)
 &\leq C\left(E_*^{3/2}
       +\bigl[R_*|(u)_{B_{R_*}}|\bigr]^3\right),
 \label{eq:C-from-E}\\
 \sup_{0<r<R_*}D_5(r)
 &\leq C\left(D_5(R_*)+E_*^{3/2}\right),
 \label{eq:D-from-E}
\end{align}
where the constants are universal.  The estimate holds for the pressure occurring in
the given suitable pair; its fixed-scale value $D_5(R_*)$ records the harmless freedom
to add a pressure constant.
\end{Lem}

\begin{proof}
Write
\[
 m(r):=(u)_{B_r}= \frac{1}{|B_r|}\int_{B_r}u\,dx.
\]
We first control the scale-invariant velocity mean.  If $0<2r<R_0$, then
H\"older's inequality and Poincar\'e's inequality on $B_{2r}$ give
\begin{align}
 |m(r)-m(2r)|
 &\leq \frac{1}{|B_r|} \int_{B_r}|u-m(2r)|\,dx\notag\\
 &\leq C r^{-5/2}\|u-m(2r)\|_{L^2(B_{2r})}\notag\\
 &\leq C r^{-3/2}\|\nabla u\|_{L^2(B_{2r})}
 \leq C E_*^{1/2}r^{-1}.
 \label{eq:mean-one-step}
\end{align}
For a given $0<r<R_*$, choose $j\geq0$ so that
$R_*/2<2^jr\leq R_*$.  Iterating \eqref{eq:mean-one-step}, and then comparing the
two comparable balls $B_{2^jr}$ and $B_{R_*}$ once more by Poincar\'e's inequality,
we obtain
\begin{align}
 |m(r)|
 &\leq |m(R_*)|+
 C E_*^{1/2}\left(R_*^{-1}+\sum_{\ell=0}^{j-1}(2^\ell r)^{-1}\right)\notag\\
 &\leq |m(R_*)|+C E_*^{1/2}r^{-1}.
\end{align}
Consequently,
\begin{equation}
 \sup_{0<r<R_*}r|m(r)|
 \leq C\left(R_*|m(R_*)|+E_*^{1/2}\right).
 \label{eq:mean-scale-bound}
\end{equation}

The five-dimensional Sobolev--Poincar\'e inequality and H\"older's inequality imply
\begin{align}
 \int_{B_r}|u-m(r)|^3\,dx
 &\leq |B_r|^{1/10}
       \|u-m(r)\|_{L^{10/3}(B_r)}^3\notag\\
 &\leq C r^{1/2}\|\nabla u\|_{L^2(B_r)}^3
 \leq C E_*^{3/2}r^2.
 \label{eq:oscillation-L3}
\end{align}
Combining \eqref{eq:mean-scale-bound} and \eqref{eq:oscillation-L3},
\begin{align*}
 C_5(r)
 &\leq Cr^{-2}\int_{B_r}|u-m(r)|^3\,dx
      +Cr^{-2}|B_r||m(r)|^3\\
 &\leq C\left(E_*^{3/2}
       +[R_*|m(R_*)|]^3\right),
\end{align*}
which proves \eqref{eq:C-from-E}.

We next estimate the pressure.  Fix $0<r\leq R_*$ and put
$a=m(2r)$.  Since $a$ is constant and $\diver u=0$,
\begin{equation}
 \partial_i\partial_j(u_iu_j)
 =\partial_i\partial_j\bigl[(u_i-a_i)(u_j-a_j)\bigr]
 \quad\text{in }\cD'(B_{2r}).
 \label{eq:pressure-subtract-mean}
\end{equation}
Choose $\chi_r\in C_c^\infty(B_{2r})$ with $\chi_r=1$ on $B_r$ and define on
$\R^5$
\[
 p_{1,r}:=\mathcal R_i\mathcal R_j
 \left(\chi_r(u_i-a_i)(u_j-a_j)\right),
 \qquad h_r:=p-p_{1,r}.
\]
By the pressure equation and \eqref{eq:pressure-subtract-mean}, $h_r$ is harmonic in
$B_r$.  The Calder\'on--Zygmund estimate and \eqref{eq:oscillation-L3}, applied on
$B_{2r}$, give
\begin{equation}
 \int_{\R^5}|p_{1,r}|^{3/2}\,dx
 \leq C\int_{B_{2r}}|u-m(2r)|^3\,dx
 \leq C E_*^{3/2}r^2.
 \label{eq:pressure-local-part}
\end{equation}

Let $0<\theta\leq1/2$.  The interior $L^{3/2}$ estimate for the harmonic function
$h_r$ yields
\begin{equation}
 \int_{B_{\theta r}}|h_r|^{3/2}\,dx
 \leq C\theta^5\int_{B_r}|h_r|^{3/2}\,dx.
 \label{eq:harmonic-pressure-decay}
\end{equation}
Using $p=p_{1,r}+h_r$, \eqref{eq:pressure-local-part}, and
\eqref{eq:harmonic-pressure-decay}, we find
\begin{align}
 D_5(\theta r)
 &\leq C(\theta r)^{-2}
       \int_{B_{\theta r}}|p_{1,r}|^{3/2}\,dx
   +C(\theta r)^{-2}
       \int_{B_{\theta r}}|h_r|^{3/2}\,dx\notag\\
 &\leq C\theta^{-2}E_*^{3/2}
   +C\theta^3D_5(r).
 \label{eq:pressure-decay-iteration}
\end{align}
Here the occurrence of $p_{1,r}$ inside
$\int_{B_r}|h_r|^{3/2}$ is absorbed once more by
\eqref{eq:pressure-local-part}.

Choose a fixed $\theta_0\in(0,1/2]$ so small that the coefficient in
\eqref{eq:pressure-decay-iteration} satisfies $C\theta_0^3\leq1/2$.  Iteration gives
\begin{equation}
 D_5(\theta_0^jR_*)
 \leq 2^{-j}D_5(R_*)+C\theta_0^{-2}E_*^{3/2}
 \quad (j\geq0).
 \label{eq:pressure-discrete-iteration}
\end{equation}
For arbitrary $0<r<R_*$, select $j$ with
$\theta_0^{j+1}R_*<r\leq\theta_0^jR_*$.  By inclusion of the centered balls,
\[
 D_5(r)
 \leq\theta_0^{-2}D_5(\theta_0^jR_*).
\]
Together with \eqref{eq:pressure-discrete-iteration}, this proves
\eqref{eq:D-from-E}.  Finally, $m(R_*)$ and $D_5(R_*)$ are finite because
$u\in H^1_{\rm loc}(B_1)$ and $p\in L^{3/2}_{\rm loc}(B_1)$.  Hence
\eqref{eq:Dirichlet-Morrey-conclusion} follows.
\end{proof}

\begin{Lem}
\label{lem:C-to-D}
Let $(u,p)$ be a suitable weak solution of \eqref{eq:NS} in $B_1\subset\R^5$ and
suppose that, for some $R_0\in(0,1)$,
\begin{equation}
 M_*:=\sup_{0<r<R_0}C_5(r)<\infty.
 \label{eq:C-star}
\end{equation}
Set $R_*:=R_0/4$.  Then
\begin{equation}
 \sup_{0<r<R_*}D_5(r)
 \leq C\bigl(D_5(R_*)+M_*\bigr),
 \label{eq:D-from-C}
\end{equation}
where $C$ is universal.
\end{Lem}

\begin{proof}
Fix $0<r\leq R_*$.  Choose $\chi_r\in C_c^\infty(B_{2r})$ with
$\chi_r=1$ on $B_r$ and define
\[
 p_{1,r}:=\mathcal R_i\mathcal R_j(\chi_r u_i u_j),
 \qquad h_r:=p-p_{1,r}.
\]
Taking the divergence of the momentum equation gives
$-\Delta p=\partial_i\partial_j(u_i u_j)$ in $B_1$.  Hence $h_r$ is harmonic in
$B_r$.  Since $2r<R_0$, the Calder\'on--Zygmund estimate and \eqref{eq:C-star} give
\begin{equation}
 \int_{\R^5}|p_{1,r}|^{3/2}\,dx
 \leq C\int_{B_{2r}}|u|^3\,dx
 \leq C M_*r^2.
 \label{eq:p1-from-C}
\end{equation}
For $0<\theta\leq1/2$, the interior estimate for harmonic functions yields
\[
 \int_{B_{\theta r}}|h_r|^{3/2}\,dx
 \leq C\theta^5\int_{B_r}|h_r|^{3/2}\,dx.
\]
Using $p=p_{1,r}+h_r$ twice and then \eqref{eq:p1-from-C}, we obtain
\begin{align}
 D_5(\theta r)
 &\leq C(\theta r)^{-2}\int_{B_{\theta r}}|p_{1,r}|^{3/2}\,dx
   +C(\theta r)^{-2}\int_{B_{\theta r}}|h_r|^{3/2}\,dx\notag\\
 &\leq C\theta^{-2}M_*+C\theta^3D_5(r)+C\theta^3M_*\notag\\
 &\leq C\theta^{-2}M_*+C\theta^3D_5(r).
 \label{eq:C-pressure-contraction}
\end{align}
Choose $\theta_0\in(0,1/2]$ so that $C\theta_0^3\leq1/2$.  Iterating
\eqref{eq:C-pressure-contraction} gives
\[
 D_5(\theta_0^jR_*)
 \leq2^{-j}D_5(R_*)+C\theta_0^{-2}M_*,
 \qquad j\geq0.
\]
If $\theta_0^{j+1}R_*<r\leq\theta_0^jR_*$, inclusion of the centered balls gives
\[
 D_5(r)\leq\theta_0^{-2}D_5(\theta_0^jR_*).
\]
This proves \eqref{eq:D-from-C}.  The starting value $D_5(R_*)$ is finite because
$p\in L^{3/2}_{\rm loc}(B_1)$.
\end{proof}

\subsection{Blow-up and completion of the regularity proofs}

\begin{proof}[Proof of Theorem \ref{thm:dirichlet-removability}]
If \eqref{eq:local-Dirichlet-bound} holds, Lemma
\ref{lem:Dirichlet-to-Morrey} gives both bounds in \eqref{eq:localMorrey} on a
possibly smaller fixed interval of radii.  If instead
\eqref{eq:local-velocity-Morrey-bound} holds, Lemma \ref{lem:C-to-D} supplies the
missing pressure bound, again after reducing the upper radius by a fixed factor.  In
either case, Theorem \ref{cor:first-singularity} implies that the origin is regular.
Since the solution is already smooth on $B_1\setminus\{0\}$, standard elliptic
bootstrapping gives smoothness throughout $B_1$.
\end{proof}

\begin{proof}[Proof of Theorem \ref{cor:first-singularity}]
Assume that the origin is singular.  The one-scale criterion in Proposition
\ref{thm:epsilon} gives
\begin{equation}
 C_5(r)\geq\eps_*
 \quad\text{for all sufficiently small }r>0.
 \label{eq:first-singularity-lower}
\end{equation}
Fix any sequence $r_k\downarrow0$ and define
\[
 u_k(x)=r_ku(r_kx),
 \qquad p_k(x)=r_k^2p(r_kx).
\]
For every fixed $R>0$, one has $r_kR<r_0$ for all sufficiently large $k$; hence
\eqref{eq:localMorrey} gives the same bounds for $(u_k,p_k)$ on $B_R$.
Lemma \ref{lem:compactness} gives, after extraction, a global suitable weak limit
$(U,P)$ satisfying the global Morrey bounds and
\begin{equation}
 \int_{B_1}|U|^3\,dx
 =\lim_{k\to\infty}C_5(r_k)\geq\eps_*.
 \label{eq:first-singularity-nonzero}
\end{equation}
It remains to show that this tangent inherits enough of the head-pressure inequality.
This is the point at which the argument differs from the pointwise compactness proof in
\cite{BGLWX2025}.

On every compact annulus $A\Subset\R^5\setminus\{0\}$, the pairs $(u_k,p_k)$ are
smooth for all large $k$.  Write
\[
 H_k=\frac{|u_k|^2}{2}+p_k,
 \qquad W_k=(H_k)_+^{3/4}.
\]
Proposition \ref{prop:smooth-RHP} and the uniform annular estimate
\eqref{eq:annular-W} show that $\{W_k\}$ is bounded in $H^1(A)$.  Lemma
\ref{lem:pressure-strong}, applied on the punctured space, improves the pressure
convergence to
\[
 p_k\to P\quad\text{strongly in }L^{3/2}_{\rm loc}
 (\R^5\setminus\{0\}).
\]
Together with $u_k\to U$ strongly in $L^3_{\rm loc}$, this gives
\[
 H_k\to H:=\frac{|U|^2}{2}+P
 \quad\text{strongly in }L^{3/2}_{\rm loc}.
\]
Since the positive part is Lipschitz and
$|a^{3/4}-b^{3/4}|^2\leq |a-b|^{3/2}$ for $a,b\geq0$, it follows that
\[
 W_k\to W:=H_+^{3/4}\quad\text{strongly in }L^2_{\rm loc}.
\]
The uniform $H^1$ bound, Rellich's theorem, and interpolation with the uniform
$L^{10/3}$ bound further give, after extraction,
\begin{equation}
 W_k\to W\quad\text{strongly in }L^3_{\rm loc},
 \qquad
 \nabla W_k\rightharpoonup\nabla W\quad\text{weakly in }L^2_{\rm loc}.
 \label{eq:Wk-stability}
\end{equation}

Define
\[
 \cF_k=\frac43W_k\nabla W_k-\frac23u_kW_k^2,
 \qquad
 \cF=\frac43W\nabla W-\frac23UW^2.
\]
The first convergence in \eqref{eq:Wk-stability} and strong $L^3$ convergence of
$u_k$ show that $\cF_k\to\cF$ in distributions: the first product is strong--weak in
$L^2\cdot L^2$, and $u_kW_k^2\to UW^2$ strongly in $L^1$.  For every nonnegative
$\phi\in C_c^\infty(\R^5\setminus\{0\})$, RHP for $(u_k,p_k)$ gives
\[
 -\int\cF_k\cdot\nabla\phi\,dx
 \geq\frac89\int\phi|\nabla W_k|^2\,dx.
\]
Passing to the limit and using weak lower semicontinuity yields the reduced RHP
inequality \eqref{eq:RHP} for $(U,P)$. Proposition \ref{Thm:RHP-sign} therefore gives $H\leq0$.  Proposition
\ref{thm:sign-rigidity} now implies $U=P=0$, contradicting
\eqref{eq:first-singularity-nonzero}.  Hence the origin is regular.  Standard elliptic
bootstrapping, combined with the assumed punctured smoothness, gives smoothness in
$B_1$.
\end{proof}

\section{The four-dimensional velocity-only rigidity: proof of Theorem \ref{thm:cubic-45}}

The four-dimensional endpoint is treated separately in this section.  After canonical
pressure reconstruction, the punctured solution is extended in $H^1_{\rm loc}$; the
borderline inner-cutoff errors are only bounded, which is nevertheless sufficient in
four dimensions.  The embedding $\dot H^1(\R^4)\hookrightarrow L^4(\R^4)$ then gives
global velocity and pressure integrability, and an outer-cutoff argument makes the
Dirichlet energy vanish.  No RHP or five-dimensional compactness argument is used.

\begin{proof}[Proof of Theorem \ref{thm:cubic-45}]
Apply Proposition \ref{prop:auto-pressure} with $n=4$ to $(U,P_0)$.  It provides a
canonical pressure $P=P^\sharp$ such that
\[
 \sup_{R>0}R^{-1}\int_{B_R}|P|^{3/2}\,dx
 \leq C\cM_{3,4}(U),
\]
the pair $(U,P)$ satisfies the stationary equations on all of $\R^4$ in the
distributional sense, and $P_0=P+c_0$ on the punctured space for a constant $c_0$.
It therefore suffices to prove $U=P=0$ for this canonical pair.  Set
\[
 M_4:=\sup_{R>0}R^{-1}\int_{B_R}|U|^3\,dx
 +\sup_{R>0}R^{-1}\int_{B_R}|P|^{3/2}\,dx.
\]
Throughout this proof, $C$ may depend on $M_4$, but never on the inner or outer
cutoff radius.

\smallskip
\noindent\emph{Step 1: extension across the puncture.}
The cubic Morrey bounds and H\"older's inequality give, for every \(r>0\),
\begin{align}
 \int_{B_r}|U|^2\,dx
 &\leq
 \left(\int_{B_r}|U|^3\,dx\right)^{2/3}|B_r|^{1/3}
 \leq Cr^2,
 \label{eq:4d-U2}\\
 \int_{B_r}|U|^3\,dx&\leq Cr,
 \label{eq:4d-U3}\\
 \int_{B_r}|P||U|\,dx
 &\leq
 \left(\int_{B_r}|P|^{3/2}\,dx\right)^{2/3}
 \left(\int_{B_r}|U|^3\,dx\right)^{1/3}
 \leq Cr.
 \label{eq:4d-PU}
\end{align}
Let \(\chi_\varepsilon\) be the inner cutoff used in the proof of Theorem
\ref{thm:smooth-punctured}.  Testing the smooth local energy equality on the punctured
space with \(\eta^2\chi_\varepsilon\), where
\(\eta\in C_c^\infty(B_{2R})\) and \(\eta=1\) on \(B_R\), gives
\begin{align}
 \int_{B_R\setminus B_{2\varepsilon}}|\nabla U|^2\,dx
 \leq C_R
 &+C\varepsilon^{-2}\int_{A_\varepsilon}|U|^2\,dx\notag\\
 &+C\varepsilon^{-1}\int_{A_\varepsilon}
 \bigl(|U|^3+|P||U|\bigr)\,dx.
 \label{eq:4d-punctured-Caccioppoli}
\end{align}
All terms on which a derivative falls on \(\eta\) are supported away from the
puncture and are included in \(C_R\).  On
\(A_\varepsilon=B_{2\varepsilon}\setminus B_\varepsilon\), the terms containing
\(\Delta\chi_\varepsilon\) or \(\nabla\chi_\varepsilon\) are bounded respectively by
the first and second inner-error terms in
\eqref{eq:4d-punctured-Caccioppoli}.
In dimension four the terms on the second line and the inner \(L^2\) term are only
\(O(1)\), rather than \(o(1)\).  More precisely,
\[
 \varepsilon^{-2}\int_{A_\varepsilon}|U|^2\le C,
 \qquad
 \varepsilon^{-1}\int_{A_\varepsilon}(|U|^3+|P||U|)\le C.
\]
Thus the right-hand side of \eqref{eq:4d-punctured-Caccioppoli} is uniform in
\(\varepsilon\).  Since the sets \(B_R\setminus B_{2\varepsilon}\) increase to
\(B_R\setminus\{0\}\), monotone convergence shows that the punctured gradient belongs
to \(L^2(B_R)\).  It remains to identify it with the whole-space distributional
gradient.  For a scalar test function \(\zeta\in C_c^\infty(B_R)\), integrate by
parts against \(\zeta\chi_\varepsilon\).  The only extra term is bounded by
\[
 \left|\int U\zeta\,\nabla\chi_\varepsilon\,dx\right|
 \le C_\zeta\|U\|_{L^2(B_{2\varepsilon})}
             \|\nabla\chi_\varepsilon\|_{L^2(A_\varepsilon)}
 \le C\varepsilon\cdot\varepsilon=C\varepsilon^2\longrightarrow0,
\]
where \eqref{eq:4d-U2} was used.  Passing to the limit in the integration-by-parts
identity proves that the punctured gradient is the distributional gradient across the
origin.  Therefore
\begin{equation}
 U\in H^1_{\rm loc}(\R^4).
 \label{eq:4d-H1}
\end{equation}

The global distributional equations have already been supplied by Proposition
\ref{prop:auto-pressure}.  It remains to verify suitability; here one should not try
to pass the local energy
identity with the same inner cutoff, because its second-derivative error need not
vanish.  Instead use the already extended weak equation.  By \eqref{eq:4d-H1} and the
four-dimensional Sobolev embedding,
\[
 U\in L^4_{\rm loc}(\R^4).
\]
Consequently \(U\otimes U\in L^2_{\rm loc}\),
\(|U|^2|\nabla U|\in L^1_{\rm loc}\), and
\(PU\in L^1_{\rm loc}\), since \(P\in L^{3/2}_{\rm loc}\) and
\(U\in L^3_{\rm loc}\).  Fix \(\phi\in C_c^\infty(\R^4)\).  On an open
neighborhood compactly containing \(\supp\phi\), mollify \(U\) and denote the
resulting smooth divergence-free fields by \(U_m\).  Then
\[
 U_m\to U\quad\text{in }H^1\cap L^4,
 \qquad U_m\to U\quad\text{in }L^3.
\]
Use \(U_m\phi\) as a vector test in the extended weak momentum equation.  The viscous
term converges by strong \(H^1\) convergence.  Since
\(U\otimes U\in L^2_{\rm loc}\), the convection term converges by H\"older with
exponents \(2,2\).  Moreover,
\(\diver(U_m\phi)=U_m\cdot\nabla\phi\to U\cdot\nabla\phi\) in \(L^3\), so the
pressure term converges against \(P\in L^{3/2}_{\rm loc}\).  Passing to the limit and
using \(\diver U=0\), one obtains
\begin{align*}
 \int \nabla U:\nabla(U\phi)
 &=\int |\nabla U|^2\phi-\int\frac{|U|^2}{2}\Delta\phi,\\
 -\int U\otimes U:\nabla(U\phi)
 &=-\int\frac{|U|^2}{2}U\cdot\nabla\phi,\\
 -\int P\,\diver(U\phi)&=-\int PU\cdot\nabla\phi.
\end{align*}
This gives, for every
\(\phi\in C_c^\infty(\R^4)\),
\begin{equation}
 \int |\nabla U|^2\phi\,dx
 =
 \int\frac{|U|^2}{2}\Delta\phi\,dx
 +\int\left(\frac{|U|^2}{2}+P\right)U\cdot\nabla\phi\,dx.
 \label{eq:4d-local-energy-equality}
\end{equation}
In particular, the extension is a global suitable weak solution.

\smallskip
\noindent\emph{Step 2: global Dirichlet energy and the pressure normalization.}
Apply \eqref{eq:4d-local-energy-equality} with a cutoff which is one on \(B_R\),
supported in \(B_{2R}\), and has first and second derivatives bounded by
\(CR^{-1}\) and \(CR^{-2}\).  Estimates
\eqref{eq:4d-U2}--\eqref{eq:4d-PU} yield
\begin{equation}
 \int_{B_R}|\nabla U|^2\,dx\leq C
 \quad\text{for every }R>0.
 \label{eq:4d-global-energy-bound}
\end{equation}
Hence \(\nabla U\in L^2(\R^4)\).  The homogeneous Sobolev inequality gives a constant
vector \(c\) such that \(U-c\in L^4(\R^4)\); this is the standard homogeneous
Sobolev lemma for a locally integrable function whose distributional gradient belongs
to \(L^2(\R^4)\).  On the other hand,
\[
 |(U)_{B_R}|
 \leq\left( \frac{1}{|B_R|}\int_{B_R}|U|^3\,dx\right)^{1/3}
 \leq CR^{-1}\longrightarrow0.
\]
The averages of \(U-c\in L^4(\R^4)\) also tend to zero, since
\[
 \left|(U-c)_{B_R}\right|
 \leq |B_R|^{-1/4}\|U-c\|_{L^4(\R^4)}\longrightarrow0.
\]
Therefore \(c=0\) and
\begin{equation}
 U\in L^4(\R^4).
 \label{eq:4d-U4}
\end{equation}

Let
\[
 \Pi:=\mathcal R_i\mathcal R_j(U_iU_j).
\]
Taking the divergence of the global momentum equation gives
\(-\Delta P=\partial_i\partial_j(U_iU_j)\) in distributions.  With the convention
\(-\Delta\mathcal R_i\mathcal R_j f=\partial_i\partial_jf\), it follows that
\(h:=P-\Pi\) is harmonic on all of \(\R^4\), while
\eqref{eq:4d-U4} gives \(\Pi\in L^2(\R^4)\).  Furthermore,
\[
 \int_{B_R}|h|^{3/2}\,dx
 \leq C\int_{B_R}|P|^{3/2}\,dx
 +C\left(\int_{B_R}|\Pi|^2\,dx\right)^{3/4}|B_R|^{1/4}
 \leq CR.
\]
For any fixed \(x_0\in\R^4\) and
\(R\geq2|x_0|+1\), one has \(B_R(x_0)\subset B_{2R}\).  The preceding estimate,
with \(2R\) in place of \(R\), therefore also gives
\(\int_{B_R(x_0)}|h|^{3/2}\leq CR\).  The interior \(L^{3/2}\)-to-\(L^\infty\)
estimate for harmonic functions now yields
\[
 |h(x_0)|
 \leq C R^{-8/3}\|h\|_{L^{3/2}(B_R(x_0))}
 \leq C R^{-8/3}(CR)^{2/3}
 =CR^{-2}\longrightarrow0.
\]
Thus \(h=0\), and therefore
\begin{equation}
 P=\Pi\in L^2(\R^4).
 \label{eq:4d-P2}
\end{equation}

\smallskip
\noindent\emph{Step 3: vanishing of the global energy.}
Use in \eqref{eq:4d-local-energy-equality} an outer cutoff which is one on \(B_R\)
and supported in \(B_{2R}\).  Since all its derivatives are supported in
\(A_R=B_{2R}\setminus B_R\),
\begin{align}
 \int_{B_R}|\nabla U|^2\,dx
 \leq C\bigg[
 &R^{-2}\int_{A_R}|U|^2\,dx
 +R^{-1}\int_{A_R}|U|^3\,dx\notag\\
 &+R^{-1}\int_{A_R}|P||U|\,dx\bigg].
 \label{eq:4d-tail-Caccioppoli}
\end{align}
By \eqref{eq:4d-U4}--\eqref{eq:4d-P2},
\begin{align*}
 R^{-2}\int_{A_R}|U|^2
 &\leq R^{-2}|A_R|^{1/2}\|U\|_{L^4(A_R)}^2
 \leq C\|U\|_{L^4(A_R)}^2,\\
 R^{-1}\int_{A_R}|U|^3
 &\leq R^{-1}|A_R|^{1/4}\|U\|_{L^4(A_R)}^3
 \leq C\|U\|_{L^4(A_R)}^3,\\
 R^{-1}\int_{A_R}|P||U|
 &\leq R^{-1}|A_R|^{1/4}
 \|P\|_{L^2(A_R)}\|U\|_{L^4(A_R)}
 \leq C\|P\|_{L^2(A_R)}\|U\|_{L^4(A_R)}.
\end{align*}
Every right-hand side tends to zero as \(R\to\infty\).  Letting \(R\to\infty\) in
\eqref{eq:4d-tail-Caccioppoli} gives \(\nabla U=0\).  Since \(U\in L^4(\R^4)\),
\(U=0\).  The equation and \eqref{eq:4d-P2} then give \(P=0\).  Since
$P_0=P+c_0$, the original pressure is the constant $c_0$.  This completes the proof
of Theorem \ref{thm:cubic-45}.
\end{proof}

\appendix
\section{Pressure reconstruction from the velocity}
\label{app:pressure}

This appendix supplies the pressure normalization used by both velocity-only rigidity
theorems.  A local/far-field Riesz-transform decomposition constructs a canonical
pressure with the correctly scaled $L^{3/2}$ Morrey bound.  A Bogovskii correction
then removes the possible momentum defect supported at the puncture, and a Liouville
argument proves that the original pressure differs from the canonical one only by a
constant. 

\begin{Prop}[Automatic pressure reconstruction]
\label{prop:auto-pressure}
Let $n\in\{4,5\}$ and let
\[
 (U,P_0)\in C^\infty(\R^n\setminus\{0\};\R^n\times\R)
\]
solve the stationary Navier--Stokes equations \eqref{eq:NS} in $\R^n\setminus\{0\}$.  Assume that
\begin{equation}
 \cM_{3,n}(U):=\sup_{R>0}R^{3-n}\int_{B_R}|U|^3\,dx=M<\infty.
 \label{eq:velocity-Morrey}
\end{equation}
Then there exists a canonical pressure $P^\sharp\in L^{3/2}_{\rm loc}(\R^n)$ such that
\begin{equation}
 -\Delta P^\sharp=\partial_i\partial_j(U_iU_j)
 \quad\text{in }\cD'(\R^n),
 \label{eq:A-pressure-Poisson}
\end{equation}
and
\begin{equation}
 \sup_{R>0}R^{3-n}\int_{B_R}|P^\sharp|^{3/2}\,dx\le C_n M.
 \label{eq:pressure-Morrey}
\end{equation}
Moreover,
\begin{equation}
 -\Delta U+\diver(U\otimes U)+\nabla P^\sharp=0,
 \qquad \diver U=0
 \quad\text{in }\cD'(\R^n),
 \label{eq:global-NS-canonical}
\end{equation}
and there exists a constant $c\in\R$ such that
\begin{equation}
 P_0=P^\sharp+c
 \quad\text{in }\R^n\setminus\{0\}.
 \label{eq:pressure-constant}
\end{equation}
In particular, $P^\sharp$ is smooth on $\R^n\setminus\{0\}$.  It is the unique
pressure representative associated with $U$ that satisfies
\eqref{eq:pressure-Morrey}.
\end{Prop}

The proof is divided into four steps. The only auxiliary ingredient not proved here is the standard Bogovskii operator (see Chapter III, Theorem III.3.1 in \cite{G2011}) on uniformly Lipschitz annuli: if $A\subset\R^n$ is such an annulus and $1<q<\infty$, then for every $g\in L^q(A)$ with zero mean there exists $\mathcal B_Ag\in W^{1,q}_0(A;\R^n)$ such that
\begin{equation}
 \diver \mathcal B_Ag=g,
 \qquad
 \|\nabla\mathcal B_Ag\|_{L^q(A)}\le C_q\|g\|_{L^q(A)},
 \label{eq:Bogovskii}
\end{equation}
where the constant is unchanged under dilation of the annulus.  We use the standard
linear construction, so if $g\in L^2(A)\cap L^3(A)$, the same field
$\mathcal B_Ag$ satisfies both the $q=2$ and $q=3$ estimates.

\underline{\bf Step 1: construction of the canonical pressure}

By H\"older's inequality and \eqref{eq:velocity-Morrey}, for every $\rho>0$,
\begin{align}
 \int_{B_\rho}|U|^2\,dx
 &\le
 \left(\int_{B_\rho}|U|^3\,dx\right)^{2/3}|B_\rho|^{1/3}\notag\\
 &\le C_n M^{2/3}\rho^{n-2}.
 \label{eq:L2-growth}
\end{align}

Let $\mathcal R_i$ be the Riesz transforms on $\R^n$, with Fourier multiplier
$-i\xi_i/|\xi|$, and set
\[
 T_{ij}:=\mathcal R_i\mathcal R_j.
\]
Repeated indices are summed throughout the construction.  With this convention,
$-\Delta T_{ij}f=\partial_i\partial_jf$.
Away from the origin, the Calder\'on--Zygmund kernel $K_{ij}$ of $T_{ij}$ is a
dimensional constant multiple of
\[
 \frac{n x_ix_j-\delta_{ij}|x|^2}{|x|^{n+2}}
\]
(with the usual principal-value normalization at the origin), and satisfies
\begin{equation}
 |K_{ij}(x)|\le C_n|x|^{-n}.
 \label{eq:kernel}
\end{equation}
The principal-value normalization is used only in the local term below.  Whenever
$x$ is separated from the support of $f$, $T_{ij}f(x)$ is represented by the ordinary
absolutely convergent kernel integral.
Fix $R>0$. For $x\in B_R$, define
\begin{equation}
 \begin{aligned}
 P_R^\sharp(x)
 &:= P_{R,\mathrm{loc}}^{\sharp} + P_{R,\mathrm{far}}^{\sharp} = 
 T_{ij}\bigl(U_iU_j\1_{B_{2R}}\bigr)(x)
 +\int_{\R^n\setminus B_{2R}}K_{ij}(x-y)U_i(y)U_j(y)\,dy.
 \end{aligned}
 \label{eq:local-pressure-definition}
\end{equation}
The far-field integral is absolutely convergent. Indeed, let
\[
 A_k:=B_{2^{k+1}R}\setminus B_{2^kR},\qquad k\ge1.
\]
For $x\in B_R$ and $y\in A_k$, one has $|x-y|\ge |y|/2$. Hence \eqref{eq:kernel} and \eqref{eq:L2-growth} imply
\begin{align}
 \int_{A_k}|K_{ij}(x-y)|\,|U(y)|^2\,dy
 &\le C_n(2^kR)^{-n}\int_{B_{2^{k+1}R}}|U|^2\,dy\notag\\
 &\le C M^{2/3}(2^kR)^{-2}.
 \label{eq:far-shell}
\end{align}
Therefore
\begin{equation}
 |P_{R,\mathrm{far}}^\sharp(x)|
 \le C M^{2/3}R^{-2},
 \qquad x\in B_R.
 \label{eq:far-pointwise}
\end{equation}

The definition is consistent as $R$ varies. If $0<R<S$ and $x\in B_R$, then $x$ is separated from $B_{2S}\setminus B_{2R}$. Therefore the singular integral of the annular piece is represented by its ordinary kernel for almost every $x\in B_R$, and
\[
 T_{ij}\bigl(U_iU_j\1_{B_{2S}\setminus B_{2R}}\bigr)(x)
 =\int_{B_{2S}\setminus B_{2R}}K_{ij}(x-y)U_i(y)U_j(y)\,dy.
\]
The annular contribution added to the local part is exactly canceled by the annular contribution removed from the far part. Hence
\[
 P_S^\sharp=P_R^\sharp
 \quad\text{a.e. on }B_R.
\]
Thus the family $\{P_R^\sharp\}_{R>0}$ defines a function
\[
 P^\sharp\in L^{3/2}_{\rm loc}(\R^n).
\]

For the local part, the $L^{3/2}$ boundedness of the Riesz transforms gives
\begin{align}
 \int_{B_R}|P_{R,\mathrm{loc}}^\sharp|^{3/2}\,dx
 &\le C\int_{B_{2R}}|U|^3\,dx
 \le C_nMR^{n-3}.
 \label{eq:local-pressure-bound}
\end{align}
For the far part, \eqref{eq:far-pointwise} yields
\begin{align}
 \int_{B_R}|P_{R,\mathrm{far}}^\sharp|^{3/2}\,dx
 &\le C\bigl(M^{2/3}R^{-2}\bigr)^{3/2}|B_R|
 \le C_nMR^{n-3}.
 \label{eq:far-pressure-bound}
\end{align}
Combining \eqref{eq:local-pressure-bound} and \eqref{eq:far-pressure-bound} proves \eqref{eq:pressure-Morrey}.

Finally, \eqref{eq:A-pressure-Poisson} follows locally.  On every ball compactly
contained in $B_R$, the shell series defining the far field and each of its spatial
derivatives converge uniformly: every derivative gives one additional power of
$|x-y|^{-1}$ in the estimate \eqref{eq:far-shell}.  Thus the far-field part in
\eqref{eq:local-pressure-definition} is harmonic in $B_R$.  If
$\varphi\in C_c^\infty(B_R)$, then
\[
 -\Delta T_{ij}f=\partial_i\partial_j f
\]
in distributions. Since $U_iU_j\1_{B_{2R}}=U_iU_j$ on $B_R$, we obtain \eqref{eq:A-pressure-Poisson}.

\underline{\bf Step 2: a pressure-free annular Caccioppoli estimate}

We first obtain a local Dirichlet bound without using any pressure estimate.

\begin{Lem}[Pressure-free annular Caccioppoli estimate]
\label{lem:pressure-free-Cacc}
For every $R>0$,
\begin{equation}
 \int_{B_{2R}\setminus B_R}|\nabla U|^2\,dx
 \le
 C R^{-2}\int_{B_{4R}\setminus B_{R/2}}|U|^2\,dx
 +C R^{-1}\int_{B_{4R}\setminus B_{R/2}}|U|^3\,dx.
 \label{eq:pressure-free-Cacc}
\end{equation}
Consequently,
\begin{equation}
 \int_{B_{2R}\setminus B_R}|\nabla U|^2\,dx
 \le C_n(M^{2/3}+M)R^{n-4}.
 \label{eq:annular-energy}
\end{equation}
When $n=5$, one also has
\begin{equation}
 \int_{B_R\setminus\{0\}}|\nabla U|^2\,dx
 \le C(M^{2/3}+M)R.
 \label{eq:ball-energy}
\end{equation}
In particular, when $n=5$, after assigning any value to $U(0)$, one has
\begin{equation}
 U\in H^1_{\rm loc}(\R^5).
 \label{eq:H1-extension}
\end{equation}
For both $n=4$ and $n=5$, the punctured gradient belongs to
$L^1_{\rm loc}(\R^n)$.
\end{Lem}

\begin{proof}
For $0\le s<t\le1$, set
\[
 D_s:=B_{(2+2s)R}\setminus\overline{B_{(1-s/2)R}}.
\]
Then
\[
 D_0=B_{2R}\setminus\overline{B_R},
 \qquad
 D_1=B_{4R}\setminus\overline{B_{R/2}}.
\]
Choose $\eta\in C_c^\infty(D_t)$ such that $0\le\eta\le1$, $\eta=1$ on $D_s$, and
\[
 |\nabla\eta|\le \frac{C}{R(t-s)}.
\]
Set
\[
 g:=\diver(\eta^2U)=2\eta U\cdot\nabla\eta.
\]
Since $\eta^2U$ is compactly supported in $D_t$,
\[
 \int_{D_t}g\,dx=0.
\]
After rescaling by $R$, the annuli $D_t$ range in a uniformly Lipschitz family: their
inner radii lie in $[1/2,1]$ and their outer radii in $[2,4]$.  Hence the constants in
the Bogovskii estimates below are uniform in $R,s,t$.
Let $w=\mathcal B_{D_t}g$. By \eqref{eq:Bogovskii}, for $q=2,3$,
\begin{equation}
 \|\nabla w\|_{L^q(D_t)}
 \le \frac{C}{R(t-s)}\|U\|_{L^q(D_t)}.
 \label{eq:Bogo-annulus}
\end{equation}
The field
\[
 \Phi:=\eta^2U-w
\]
is divergence free and belongs to
$W^{1,2}_{0,\sigma}(D_t)\cap W^{1,3}_{0,\sigma}(D_t)$.  Approximate it by smooth
compactly supported solenoidal fields in $W^{1,3}_0(D_t)$; since $D_t$ has finite
measure, this convergence also holds in $W^{1,2}_0(D_t)$.  The viscous and convection
pairings pass to the limit because
$\nabla U\in L^2(D_t)$ and $U\otimes U\in L^{3/2}(D_t)$.  Thus $\Phi$ is an
admissible pressure-free test field in the smooth punctured equation.  We obtain
\begin{align}
 \int_{D_t}\eta^2|\nabla U|^2\,dx
 ={}&-2\int_{D_t}\eta\nabla U:(U\otimes\nabla\eta)\,dx
 +\int_{D_t}\nabla U:\nabla w\,dx\notag\\
 &+\int_{D_t}U\otimes U:\nabla(\eta^2U)\,dx
 -\int_{D_t}U\otimes U:\nabla w\,dx.
 \label{eq:Cacc-identity}
\end{align}
Since $\diver U=0$,
\begin{equation}
 \int_{D_t}U\otimes U:\nabla(\eta^2U)\,dx
 =\int_{D_t}\eta|U|^2U\cdot\nabla\eta\,dx.
 \label{eq:conv-identity}
\end{equation}
Using Young's inequality, \eqref{eq:Bogo-annulus}, and \eqref{eq:conv-identity}, one obtains
\begin{align*}
 2\left|\int\eta\nabla U:(U\otimes\nabla\eta)\right|
 &\le \frac14\int\eta^2|\nabla U|^2
 +\frac{C}{R^2(t-s)^2}\int_{D_t}|U|^2,\\
 \left|\int\nabla U:\nabla w\right|
 &\le \theta\int_{D_t}|\nabla U|^2
 +\frac{C_\theta}{R^2(t-s)^2}\int_{D_t}|U|^2,\\
 \left|\int U\otimes U:\nabla(\eta^2U)\right|
 &\le \frac{C}{R(t-s)}\int_{D_t}|U|^3,\\
 \left|\int U\otimes U:\nabla w\right|
 &\le \|U\otimes U\|_{L^{3/2}(D_t)}
       \|\nabla w\|_{L^3(D_t)}
 \le \frac{C}{R(t-s)}\int_{D_t}|U|^3.
\end{align*}
Write
\[
 E(s):=\int_{D_s}|\nabla U|^2\,dx
\]
and
\[
 A_R:=R^{-2}\int_{D_1}|U|^2\,dx
 +R^{-1}\int_{D_1}|U|^3\,dx.
\]
The first estimate is absorbed into the left-hand side.  After this absorption, choose
$\theta>0$ sufficiently small and use $0<t-s\le1$.  Since
$\eta=1$ on $D_s$ and $\supp\eta\subset D_t$, the preceding estimates imply
\begin{equation}
 E(s)\le \frac18 E(t)+\frac{C}{(t-s)^2}A_R.
 \label{eq:hole-filling}
\end{equation}
Set $s_j=1-2^{-j}$. Iterating \eqref{eq:hole-filling} gives
\[
 E(0)
 \le 8^{-N}E(s_N)
 +CA_R\sum_{j=0}^{N-1}8^{-j}2^{2(j+1)}.
\]
Since $E(s_N)\le E(1)<\infty$ and
\[
 \sum_{j=0}^{\infty}8^{-j}2^{2(j+1)}<\infty,
\]
letting $N\to\infty$ proves \eqref{eq:pressure-free-Cacc}.

Now \eqref{eq:L2-growth} and \eqref{eq:velocity-Morrey} imply
\[
 R^{-2}\int_{B_{4R}}|U|^2\,dx\le C_nM^{2/3}R^{n-4},
 \qquad
 R^{-1}\int_{B_{4R}}|U|^3\,dx\le CMR^{n-4},
\]
which gives \eqref{eq:annular-energy}.  If $n=5$, summing
\eqref{eq:annular-energy} over the dyadic annuli
\[
 B_{2^{-j}R}\setminus B_{2^{-j-1}R},\qquad j\ge0,
\]
yields \eqref{eq:ball-energy}.

For either $n=4$ or $n=5$, H\"older's inequality on each dyadic annulus gives
\[
 \int_{B_{2^{-j}R}\setminus B_{2^{-j-1}R}}|\nabla U|\,dx
 \le C R^{n-2}2^{-j(n-2)}.
\]
Indeed, the square root of \eqref{eq:annular-energy} contributes
$C(2^{-j}R)^{(n-4)/2}$ and the square root of the annular volume contributes
$C(2^{-j}R)^{n/2}$.  The geometric series is summable, proving
$\nabla U\in L^1_{\rm loc}(\R^n\setminus\{0\})$ with an integrable extension across
the origin.

It remains to identify the punctured derivative with the distributional derivative
across the origin.  Let $\chi_\varepsilon$ be radial, equal to $0$ on
$B_\varepsilon$ and $1$ outside $B_{2\varepsilon}$, with
$|\nabla\chi_\varepsilon|\le C\varepsilon^{-1}$. For
$\varphi\in C_c^\infty(B_R)$, the function $\varphi\chi_\varepsilon$ is supported
away from the origin, so for each component of $U$,
\[
 \int U\,\partial_i(\varphi\chi_\varepsilon)\,dx
 =-\int (\partial_iU)\varphi\chi_\varepsilon\,dx.
\]
The only additional term is
\[
 \int U\varphi\,\partial_i\chi_\varepsilon\,dx.
\]
By H\"older's inequality and \eqref{eq:velocity-Morrey},
\[
 \left|\int U\varphi\,\partial_i\chi_\varepsilon\,dx\right|
 \le C_\varphi\varepsilon^{-1}\int_{A_\varepsilon}|U|\,dx
 \le C_{\varphi,M}\varepsilon^{n-2}\longrightarrow0.
\]
Letting $\varepsilon\downarrow0$ proves that the punctured weak gradient is the
whole-space distributional gradient.  Thus $U\in W^{1,1}_{\rm loc}(\R^n)$ in both
dimensions.  When $n=5$, \eqref{eq:L2-growth} and \eqref{eq:ball-energy} improve this
to \eqref{eq:H1-extension}.
\end{proof}

\underline{\bf Step 3: removal of the momentum defect at the origin}

We first extend the divergence equation across the origin.  If
$\zeta\in C_c^\infty(\R^n)$, then testing $\diver U=0$ on the punctured space with
$\chi_\varepsilon\zeta$ produces only the error
\[
 \int_{A_\varepsilon}\zeta U\cdot\nabla\chi_\varepsilon\,dx.
\]
By the velocity Morrey bound,
\begin{align*}
 \varepsilon^{-1}\int_{A_\varepsilon}|U|\,dx
 &\le C\varepsilon^{-1}
 \left(\int_{B_{2\varepsilon}}|U|^3\,dx\right)^{1/3}
 |B_{2\varepsilon}|^{2/3}\\
 &\le C M^{1/3}\varepsilon^{n-2}\longrightarrow0.
\end{align*}
Consequently $\diver U=0$ in $\cD'(\R^n)$.

We next prove the pressure-free projected equation on the whole space. Let
\[
 \phi\in C_c^\infty(\R^n;\R^n),
 \qquad \diver\phi=0.
\]
Choose the inner cutoff $\chi_\varepsilon$ as above and let
\[
 A_\varepsilon:=B_{2\varepsilon}\setminus\overline{B_\varepsilon},
 \qquad
 g_\varepsilon:=\diver(\chi_\varepsilon\phi)
 =\nabla\chi_\varepsilon\cdot\phi.
\]
Since $\chi_\varepsilon=0$ on the inner boundary and $\chi_\varepsilon=1$ on the outer boundary,
\[
 \int_{A_\varepsilon}g_\varepsilon\,dx
 =\int_{\partial B_{2\varepsilon}}\phi\cdot n\,dS
 =\int_{B_{2\varepsilon}}\diver\phi\,dx=0.
\]
Let
\[
 w_\varepsilon:=\mathcal B_{A_\varepsilon}g_\varepsilon
\]
and extend $w_\varepsilon$ by zero outside $A_\varepsilon$. By scale invariance of the Bogovskii estimate,
\begin{equation}
 \|\nabla w_\varepsilon\|_{L^q(A_\varepsilon)}
 \le C_\phi\varepsilon^{n/q-1},
 \qquad q=2,3.
 \label{eq:w-eps}
\end{equation}
Define
\[
 \phi_\varepsilon:=\chi_\varepsilon\phi-w_\varepsilon.
\]
Then $\diver\phi_\varepsilon=0$, $\phi_\varepsilon$ vanishes near the origin, and $\phi_\varepsilon=\phi$ outside $B_{2\varepsilon}$. Moreover,
\begin{equation}
 \|\nabla(\phi-\phi_\varepsilon)\|_{L^2(B_{2\varepsilon})}
 \le C_\phi\varepsilon^{n/2-1},
 \label{eq:test-L2}
\end{equation}
and
\begin{equation}
 \|\nabla(\phi-\phi_\varepsilon)\|_{L^3(B_{2\varepsilon})}
 \le C_\phi\varepsilon^{n/3-1}.
 \label{eq:test-L3}
\end{equation}
Indeed, $\phi-\phi_\varepsilon=(1-\chi_\varepsilon)\phi+w_\varepsilon$ is supported
in $B_{2\varepsilon}$, and the two displayed bounds follow by combining
\eqref{eq:w-eps} with
$\|\nabla\chi_\varepsilon\|_{L^q(A_\varepsilon)}
\le C\varepsilon^{n/q-1}$.
The field $\phi_\varepsilon$ is an admissible solenoidal weak test field in the punctured equation. Hence
\[
 \int \nabla U:\nabla\phi_\varepsilon\,dx
 -\int U\otimes U:\nabla\phi_\varepsilon\,dx=0.
\]
Here admissibility follows by approximating $\phi_\varepsilon$ on the annulus by
smooth compactly supported solenoidal fields simultaneously in $W^{1,2}$ and
$W^{1,3}$, exactly as in Step~2.
The annular energy estimate \eqref{eq:annular-energy} gives
\[
 \|\nabla U\|_{L^2(A_\varepsilon)}
 \le C_M\varepsilon^{(n-4)/2}.
\]
On $A_\varepsilon$, this estimate and \eqref{eq:test-L2} give an error bounded by
$C_{\phi,M}\varepsilon^{n-3}$.  On $B_\varepsilon$, no derivative falls on the
cutoff; the dyadic $L^1$ estimate from Step~2 gives
$\int_{B_\varepsilon}|\nabla U|\le C_M\varepsilon^{n-2}$.  Therefore
\begin{align}
 \left|\int_{B_{2\varepsilon}}\nabla U:\nabla(\phi-\phi_\varepsilon)\,dx\right|
 &\le C_{\phi,M}\varepsilon^{n-3}.
 \label{eq:viscous-error}
\end{align}
On the other hand, \eqref{eq:velocity-Morrey} gives
\[
 \|U\otimes U\|_{L^{3/2}(B_{2\varepsilon})}
 \le \left(\int_{B_{2\varepsilon}}|U|^3\,dx\right)^{2/3}
 \le C_nM^{2/3}\varepsilon^{2(n-3)/3}.
\]
Together with \eqref{eq:test-L3}, this yields
\begin{equation}
 \left|\int_{B_{2\varepsilon}}U\otimes U:\nabla(\phi-\phi_\varepsilon)\,dx\right|
 \le C_{\phi,M}\varepsilon^{n-3}.
 \label{eq:convection-error}
\end{equation}
All differences between the equation tested by $\phi$ and by
$\phi_\varepsilon$ are supported in $B_{2\varepsilon}$ and are exactly the two
errors estimated in \eqref{eq:viscous-error}--\eqref{eq:convection-error}.  Letting
$\varepsilon\downarrow0$ therefore gives
\begin{equation}
 \int_{\R^n}\nabla U:\nabla\phi\,dx
 -\int_{\R^n}U\otimes U:\nabla\phi\,dx=0
 \label{eq:projected-equation}
\end{equation}
for every $\phi\in C_c^\infty(\R^n;\R^n)$ with $\diver\phi=0$.  The first
integral is legitimate also in dimension four because Step~2 proved
$\nabla U\in L^1_{\rm loc}$.

\underline{\bf Step 4: identification of the pressure and uniqueness of the normalization}

Define the vector-valued distribution
\begin{equation}
 \mathcal F:=-\Delta U+\diver(U\otimes U)+\nabla P^\sharp
 \quad\text{in }\cD'(\R^n).
 \label{eq:defect-F}
\end{equation}
By \eqref{eq:projected-equation},
\[
 \langle\mathcal F,\phi\rangle=0
\]
for every divergence-free $\phi\in C_c^\infty(\R^n;\R^n)$. The distributional de Rham theorem \cite[Lemma~III.1.1]{G2011} therefore gives a scalar distribution $Q$ such that
\[
 \mathcal F=\nabla Q.
\]
Furthermore, \eqref{eq:A-pressure-Poisson} and $\diver U=0$ imply
\[
 \diver\mathcal F
 =\partial_i\partial_j(U_iU_j)+\Delta P^\sharp=0.
\]
Hence
\[
 \Delta Q=0
\]
in distributions.  Weyl's lemma makes $Q$ a smooth harmonic function; in particular,
every component of $\mathcal F=\nabla Q$ is harmonic.

The Morrey bounds imply polynomial growth. For $R\ge1$,
\begin{align*}
 \int_{B_R}|U|\,dx&\le C_nM^{1/3}R^{n-1},\\
 \int_{B_R}|U|^2\,dx&\le C_nM^{2/3}R^{n-2},\\
 \int_{B_R}|P^\sharp|\,dx&\le C_nM^{2/3}R^{n-2}.
\end{align*}
Thus $U$, $U\otimes U$, and $P^\sharp$ define tempered distributions, and so does
$\mathcal F$.  Since each component of $\mathcal F$ is both harmonic and tempered, it
is a harmonic polynomial.

Fix $\psi\in C_c^\infty(\R^n;\R^n)$ and set
\[
 \psi_R(x):=\psi(x/R),\qquad R\ge1.
\]
If $\supp\psi\subset B_L$, then
\begin{align}
 |\langle-\Delta U,\psi_R\rangle|
 &\le \|U\|_{L^3(B_{LR})}\|\Delta\psi_R\|_{L^{3/2}}
 \le C_{\psi,L,M}R^{n-3},
 \label{eq:growth1}\\
 |\langle\diver(U\otimes U),\psi_R\rangle|
 &\le \|U\otimes U\|_{L^{3/2}(B_{LR})}\|\nabla\psi_R\|_{L^3}
 \le C_{\psi,L,M}R^{n-3},
 \label{eq:growth2}\\
 |\langle\nabla P^\sharp,\psi_R\rangle|
 &\le \|P^\sharp\|_{L^{3/2}(B_{LR})}\|\diver\psi_R\|_{L^3}
 \le C_{\psi,L,M}R^{n-3}.
 \label{eq:growtg3}
\end{align}
Therefore
\begin{equation}
 |\langle\mathcal F,\psi_R\rangle|\le C_{\psi,M}R^{n-3}.
 \label{eq:defect-growth}
\end{equation}
If $\mathcal F$ were a nonzero polynomial of degree $d\ge0$, write
$\mathcal F_d$ for its top homogeneous part and choose $\psi$ so that
$\int\mathcal F_d\cdot\psi\ne0$.  The change of variables $x=Ry$ then gives
\[
 \langle\mathcal F,\psi_R\rangle
 =c_\psi R^{n+d}+O(R^{n+d-1}),
 \qquad c_\psi\ne0,
\]
contradicting \eqref{eq:defect-growth}. Hence $\mathcal F=0$, which proves \eqref{eq:global-NS-canonical}.

On $\R^n\setminus\{0\}$, both $P_0$ and $P^\sharp$ are pressures associated with the same velocity $U$. Subtracting the two momentum equations gives
\[
 \nabla(P_0-P^\sharp)=0.
\]
Since $\R^n\setminus\{0\}$ is connected, \eqref{eq:pressure-constant} follows.
Moreover, \eqref{eq:global-NS-canonical} and the smoothness of $U$ away from the
origin show that $\nabla P^\sharp$ is smooth there; hence $P^\sharp$ is smooth on
$\R^n\setminus\{0\}$.

Finally, suppose that both $P^\sharp$ and $P^\sharp+c$ satisfy a bound of the form \eqref{eq:pressure-Morrey}. By Minkowski's inequality,
\[
 |c|\,|B_R|^{2/3}
 =\|c\|_{L^{3/2}(B_R)}
 \le
 \|P^\sharp+c\|_{L^{3/2}(B_R)}
 +\|P^\sharp\|_{L^{3/2}(B_R)}
 \le C R^{2(n-3)/3}.
\]
Since $|B_R|^{2/3}\simeq R^{2n/3}$, we obtain
\[
 |c|\le CR^{-2}.
\]
Letting $R\to\infty$ gives $c=0$. This proves the uniqueness of the Morrey normalization and completes the proof of Proposition~\ref{prop:auto-pressure}.

\section*{Declarations}
	\begin{itemize}
		\item \textbf{Acknowledgments} 
		W. Wang was supported by National Key R\&D Program of China (No.2023YFA1009200) and NSFC under grant 12471219.
		\item \textbf{Conflict of interest} The authors declare that they have no conflict of interest.
		\item \textbf{Data Availability} Data sharing is not applicable to this article as no datasets were generated or analyzed during the current study.
	\end{itemize}

\end{document}